\documentclass[11pt]{article}
\usepackage{amsmath}
\usepackage{mathtools}
\usepackage{amsthm}
\usepackage{amsfonts}
\usepackage{amssymb}
\usepackage{hyperref}
\usepackage{graphicx}
\usepackage{enumitem}
\usepackage{float}
\usepackage[normalem]{ulem}
\usepackage{tcolorbox}
\usepackage[style=alphabetic]{biblatex}
\usepackage{tikz-cd}
\usepackage{caption}
\usepackage{subcaption}

\usepackage[margin=1.5in]{geometry}

\newtheorem{thm}{Theorem}[section]
\newtheorem{cor}[thm]{Corollary}
\newtheorem{defn}[thm]{Definition}
\newtheorem{notation}[thm]{Notation}
\newtheorem{lemma}[thm]{Lemma}
\newtheorem{prop}[thm]{Proposition}
\newtheorem{conjecture}[thm]{Conjecture}

\newcommand{\R}{\mathbb{R}}
\newcommand{\T}{\mathbb{T}}
\newcommand{\Z}{\mathbb{Z}}
\newcommand{\Q}{\mathbb{Q}}

\newcommand{\N}{\mathbb{N}}

\newcommand{\supp}{\text{supp}}
\newcommand{\Leb}{\text{Leb}}

\begin{document}

\author{Margaret Brown}
\title{Expanding Maps on Flowers, Interval Exchange Transformations, and Ergodic Optimization}
\date{}
\maketitle

\begin{abstract}
    In this paper, we discuss expanding maps on a class of invariant sets called flowers. We show that any set contained in a flower has at most linear complexity, and we present a relationship between flowers and a special class of interval exchange transformations. This extends work of Bullett and Sentenac, who showed that any Sturmian system may be embedded into the circle as a doubling-invariant subset that is contained in a half circle.

    Flowers were first introduced in the context of ergodic optimization, as candidate sets for supporting maximizing measures. We discuss the relationship to ergodic optimization, and present numerical results that support the conjecture that trigonometric polynomials are maximized on flowers.
\end{abstract}

\section{Introduction}\label{sec:introduction}

Sturmian subshifts appear in many problems in dynamics, in particular in ergodic optimization, and they provide an initial motivation for the work in this paper. Let $(X_2^+,\sigma)$ be the one-sided shift on two symbols. A Sturmian subshift is a closed $\sigma$ invariant set $X\subset X_2^+$ that is recurrent and whose word complexity $C(n)$ is bounded above by $n+1$. Such shifts have been studied extensively in symbolic dynamics, for example in \cite{lothaire}, \cite{pyth}. Note that in the literature, Sturmians are often defined to be infinite systems that have complexity exactly equal to $n+1$, which was shown in \cite{morsehedlund} to be the least possible complexity for an infinite sequence. However we want to allow periodic subsystems as well, hence why we allow complexity to be simply bounded above by $n+1$.

It is known that Sturmian subshifts are uniquely ergodic. Let $\mu$ be the unique ergodic measure of a given Sturmian subshift, and define the number $\rho:=\mu(C_1)$ where $C_1$ is the cylinder set fixing $1$ in the first position, that is \[C_1=\Big\{\{1x_1x_2x_3...\} | x_i\in \{0,1\} \text{ for } i\geq 1 \Big\}.\]
Then, the subshift is a symbolic coding of the rotation $x\mapsto x+\rho$ under the partition $[0,1-\rho),[1-\rho,1)$. If $\rho \in \Q$, the Sturmian system is periodic and hence finite, otherwise it is infinite. Sturmian subshifts are a one-parameter family parametrized by the rotation number $\rho$.

In their paper \cite{bullet},  Bullett and Sentenac prove the following characterization of Sturmians. Let $E_2:\T \to \T$ be the doubling map $E_2(x)=2x \mod 1$ on the circle $\T=\R/\Z$. An \emph{invariant set} of the doubling map is a set $X$ such that $E_2(X)=X$. Bullett and Sentenac showed that a closed invariant set $X \subset \T$ is contained in a half circle $[c,c+\frac{1}{2}]$ for some $c$ if and only if the set is an embedding of a Sturmian subshift. They additionally show that invariant sets that are contained in a half circle are exactly the invariant sets for which cyclic order is preserved under $E_2$. The preservation of cyclic order allows for an alternative definition of the rotation number $\rho$ that parametrizes the Sturmian (see Lemma 1 and Proposition 1 in \cite{bullet}). 

Furthermore, the authors specify that the relation between the rotation number $\rho$ and the endpoint $c$ of the half circle is given by a devil's staircase map (Figure 1 in \cite{bullet}) and provide an algorithm for calculating $c$ given $\rho$. They show that if $\rho$ is irrational, $X$ is a Cantor set, and there is a unique half circle containing it. If $\rho$ is rational, $X$ is a unique periodic orbit, and there are a family of half circles containing it. Additionally, they show that this characterizes all invariant sets contained inside half circles: any half circle contains one unique invariant set under the doubling map.

In this paper, we will extend the work of Bullett and Sentenac by considering expanding maps on a special kind of set called a flower. Note that for a non-empty closed subset $X$ of $\T$ to contain a non-empty closed $E_2$ invariant set, a sufficient condition is that for all $x\in X$, there is at least one preimage of $x$ that falls within $X$, that is $E_2^{-1}(x) \cap X \neq \emptyset$. A half circle certainly satisfies this property. To generalize this, suppose that $X \subset \T$ is a union of finitely many closed non-degenerate intervals and satisfies the following two properties:
\[ X \cup (X+\frac{1}{2})=\T \quad \text{and} \quad X \cap (X+\frac{1}{2})=\partial X.\]
Then every $x\in X$ has at least one preimage under $E_2$ that falls in $X$. Additionally note that there are only finitely many points that have more than one preimage in $X$. Such sets are called \emph{flowers} for the map $E_2$. A similar definition can be made for general expanding orientation preserving maps with degree $d\geq2$. Flowers were defined in \cite{brem05}, \cite{brem06} and \cite{flattening} and were studied in the context of problems related to ergodic optimization. 

Let $E_d:\T \to \T$ be the expanding map $E_d(x)=dx \mod 1$. In this paper, our first result is that an invariant set of $E_d$ that is contained in a flower has at most linear complexity. A closed $E_d$ invariant subset of $\T$ is called a subsystem of $E_d$.

\begin{thm}[Theorem \ref{thm:complexitybound} in the main text]\label{thm:2}
    Any subsystem of $E_d$ that is contained in a flower with $p$ petals has at most linear complexity; more precisely: $C(n) \leq pn+k$ for some constant $k$.
\end{thm}

Theorem \ref{thm:2} generealizes Bullett and Sentenac beyond the Sturmian case. We will also show that, with a mild additional assumption, being contained in a flower is equivalent to the Sturmian-like property as introduced recently in \cite{gaoshen2}. 

Next we prove that for $E_2$, any invariant set that is contained in a flower can be realized as a coding of a special type of interval exchange transformation (IET). Interval exchange transformations are a type of dynamical system that generalize rotations. They are piecewise linear maps on the interval $[0,1)$, defined by cutting $[0,1)$ into finitely many subintervals of given lengths and then permuting the subintervals according to a given permutation. In section \ref{sec:IETs} we discuss basics of IETs and in Definition \ref{def:deckshuffler} we introduce the particular class of IETs that we use in the following theorem.

\begin{thm}[Theorems \ref{makesiet} and \ref{thm:containedinflower} in the main text]\label{thm:1}
    A subsystem of $E_2$ is contained in a flower $F$ with $2m-1$ petals and with $0 \notin F$ if and only if there is a right continuous increasing map that intertwines the dynamics of $E_2|_F$ with a deck shuffler IET with $2m$ intervals.
\end{thm}

Theorem \ref{thm:1} is also an extension of the Sturmian case proved in \cite{bullet} that an invariant set contained in a half circle is an embedding of a rotation. The condition $0\notin F$ is necessary because deck-shuffler IETs do not have fixed points. In section \ref{section:flowertoiet} we will explain why $0\notin F$ is a mild condition.

The relationship described in Theorem \ref{thm:1} has measure theoretic meaning. Specifically, we will show that given a flower $F$ together with a measure $\mu$ supported in $F$, the cumulative distribution function of $\mu$ is a measure-theoretic isomorphism with a deck shuffler IET. Conversely, we will show that a special coding of a deck shuffler IET maps the interval into a flower, and that the pushforward of Lebesgue under that coding is $E_2$ invariant.

Finally, we will discuss of the relevance of flowers and low-complexity sets to the study of ergodic optimization. Ergodic optimization studies properties of maximizing measures in the following sense:
Let $S:X\to X$ be a topological dynamical system and let $\mathcal{M}_S$ be the set of all $S$-invariant probability measures. For a observable function $f:X\to \R$, consider\[
\beta(f)=\sup_{\mu \in \mathcal{M}_S}\int_X f d\mu .
\] Any measures $\mu$ that attain the maximum $\beta(f)$ are called maximizing measures for the system. See \cite{bochi} and \cite{jenk06} for more information about ergodic optimization. 

In general, it is expected that maximizing measures have low complexity. For example, Contreras proved in \cite{contreras} that maximizing measures of Lipschitz observables for expanding maps are typically periodically supported. Similarly, recently Huang et al. proved typical periodic optimization for Lipschitz functions for a broader class of subshifts in \cite{huangetal}. Both papers use a topological notion of typicality. As a specific example, in \cite{bousch}, Bousch shows that for the dynamics $E_2:\T\to \T$ and the family of observables $f(x)=a\cos(2\pi x)+b\sin(2\pi x)$, all maximizing measures are supported by Sturmian orbits, that is orbits that lift to Sturmian subshifts. The observables Bousch studied are called degree one trigonometric polynomials, and in general, a \emph{degree $n$ trigonometric polynomial} is a function of the form $\sum_{k=1}^n(a_1\cos(k2\pi x)+b_1\sin(k2\pi x))$. In section \ref{sec:ergodicopt}, we present numerical evidence that for any degree $n$ trigonometric polynomial, its maximizing measures must have support contained in a flower with $n$ petals.

In \cite{gaoshen2}, Gao et al use the Sturmian-like property to prove typicality (in both topological and probabilistic senses) of periodic optimization. A key step in their proof is showing that typically, maximizing measures are flower-supported (that is, Sturmian-like).

The paper is organized as follows. Section \ref{sec:flowers} defines flowers, proves Theorem \ref{thm:2}, and shows the relationship between flower-supported and Sturmian-like. In Section \ref{sec:IETs}, we define and discuss basics of IETs. In Section \ref{sec:flowerplusiet} we define the subclass of IETs called deck shuffler IETs and prove Theorem \ref{thm:1}, as well as some additional details of the relationship. There are explicit examples included. Finally, in Section \ref{sec:ergodicopt}, we discuss ergodic optimization and present numerical results.

\section{Flowers}\label{sec:flowers}
\subsection{Basic Definitions}\label{subsec:flowers}

Flowers were first defined by Brémont in \cite{brem05} and \cite{brem06}. The definition was later refined by Harriss and Jenkinson in \cite{flattening}, and theirs is the definition that we reproduce here.

Let $d\geq 2$ be an integer. Let $E_d:\T \to \T$ be the expanding map $E_d(x)=dx \mod 1$. In all three papers \cite{brem05}, \cite{brem06}, and \cite{flattening}, flowers are defined for any orientation-preserving expanding map of the circle. However, any orientation-preserving expanding map of the circle with degree $d$ is topologically conjugate to the map $E_d$ (see Theorem 2.4.6 in \cite{kh}). Therefore, we will deal directly with $E_d$.

\begin{defn}\label{def:preimageselector}
    A \emph{preimage selector} for $E_d$ is a map $\eta: \T \to \T$ such that $\eta(x) \in E_d^{-1}(x)$ for all $x \in \T$ and $\eta(x)$ has only finitely many jump discontinuities and no other discontinuities. Jump discontinuities are $z$ such that both $\lim_{x\to z^+}\eta(x)$ and $\lim_{x\to z^-}\eta(x)$ exist, they are not equal, and one of them is equal to $\eta(z)$.
\end{defn}
 The name ``preimage selector'' tells us everything we need to know: $\eta$ is a function that breaks the circle into some finite number of subintervals and chooses one of the $d$ preimage branches of $E_d$ to follow on each subinterval, although it must do so in a way that creates only jump discontinuities, no removable discontinuities. Note that $\eta$ is automatically injective. Additionally, the composition $E_d\circ \eta$ is the identity map on all of $\T$ while $\eta \circ E_d$ is the identity map only on $\eta(\T)$.

\begin{defn}
   The closure of the image $F:=\overline{\eta(\T)}$ is called a \emph{flower} for the map $E_d$. For $\eta$ with $p$ discontinuities, $F$ is a \emph{$p$-flower}. It has $p$ connected components, which are closed intervals, and are called \emph{petals}.   
\end{defn}
See Figure \ref{fig:flowerexs} for examples of flowers. A given preimage selector determines a unique flower. On the other hand, for a given flower $F$ there may be multiple preimage selectors $\eta$ with $\overline{\eta(\T)}=F$; however, all such preimage selectors will differ only in whether each jump discontinuity is left or right continuous, not in the location of the jump discontinuities. We avoid the ambiguity by making a choice that all preimage selectors will be right continuous.

A flower for $E_d$ has total Lebesgue measure equal to $1/d$.  For the doubling map $E_2$, a flower's petals are mutually non-antipodal except for at the end points of the petals. Similarly, for $E_d$ each petal $P$ of a flower will have the property that the interior of $P+\frac{i}{d}$ for $i \in \{1,...,d-1 \}$ does not intersect any other petal.
There may be some restrictions on the types of flowers that can appear for a fixed $E_d$. For example, for the doubling map $E_2$, in order to ensure that $\eta$ has only jump discontinuities and no removable discontinuities, any flower must have an odd number of petals. Note also that for $E_2$, a 1-flower is simply a half circle.
\begin{figure}
\centering
\begin{subfigure}{.5\textwidth}
  \centering
  \includegraphics[width=.7\linewidth]{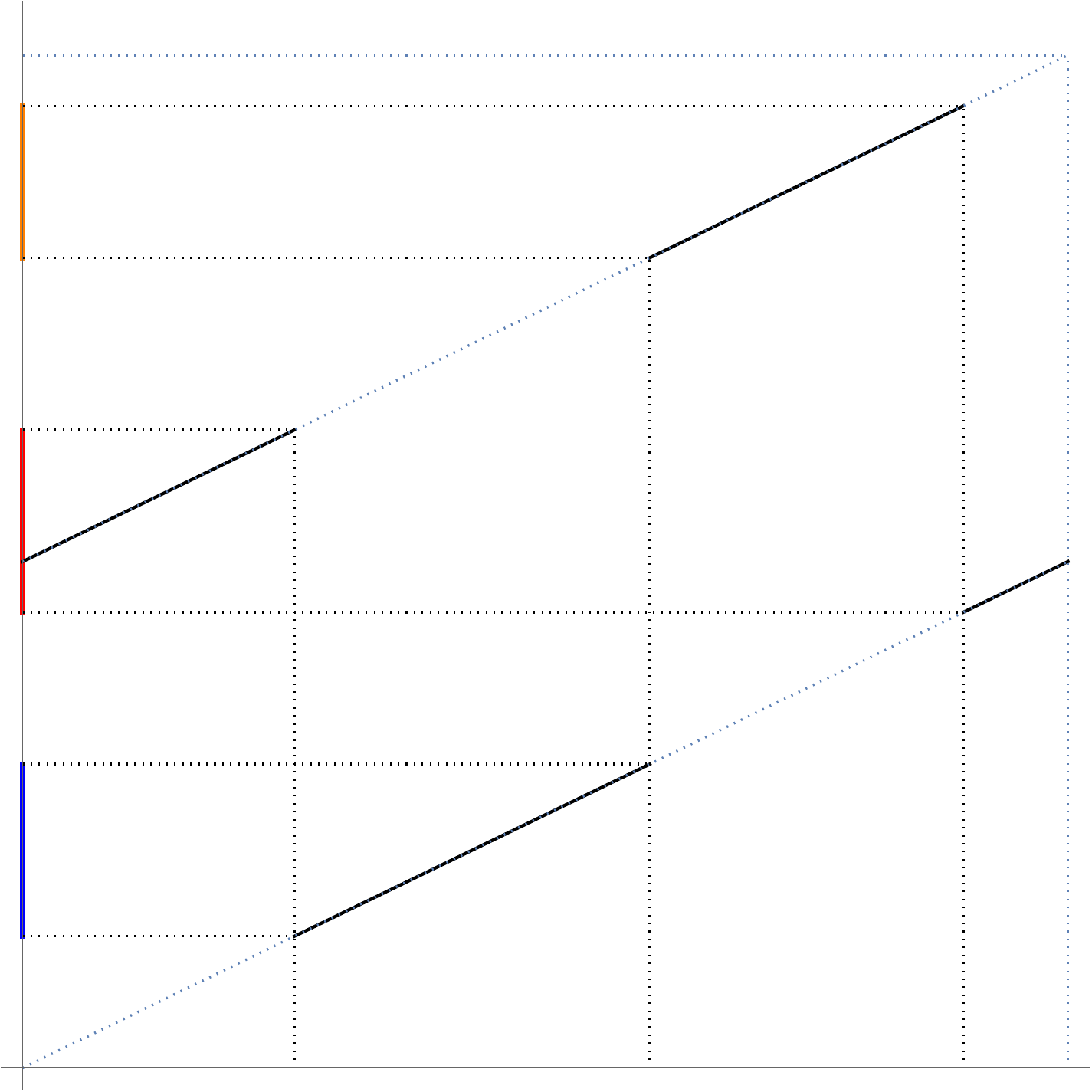}
  \caption{a 3-flower for $E_2$}
  \label{fig:sub1}
\end{subfigure}%
\begin{subfigure}{.5\textwidth}
  \centering
  \includegraphics[width=.7\linewidth]{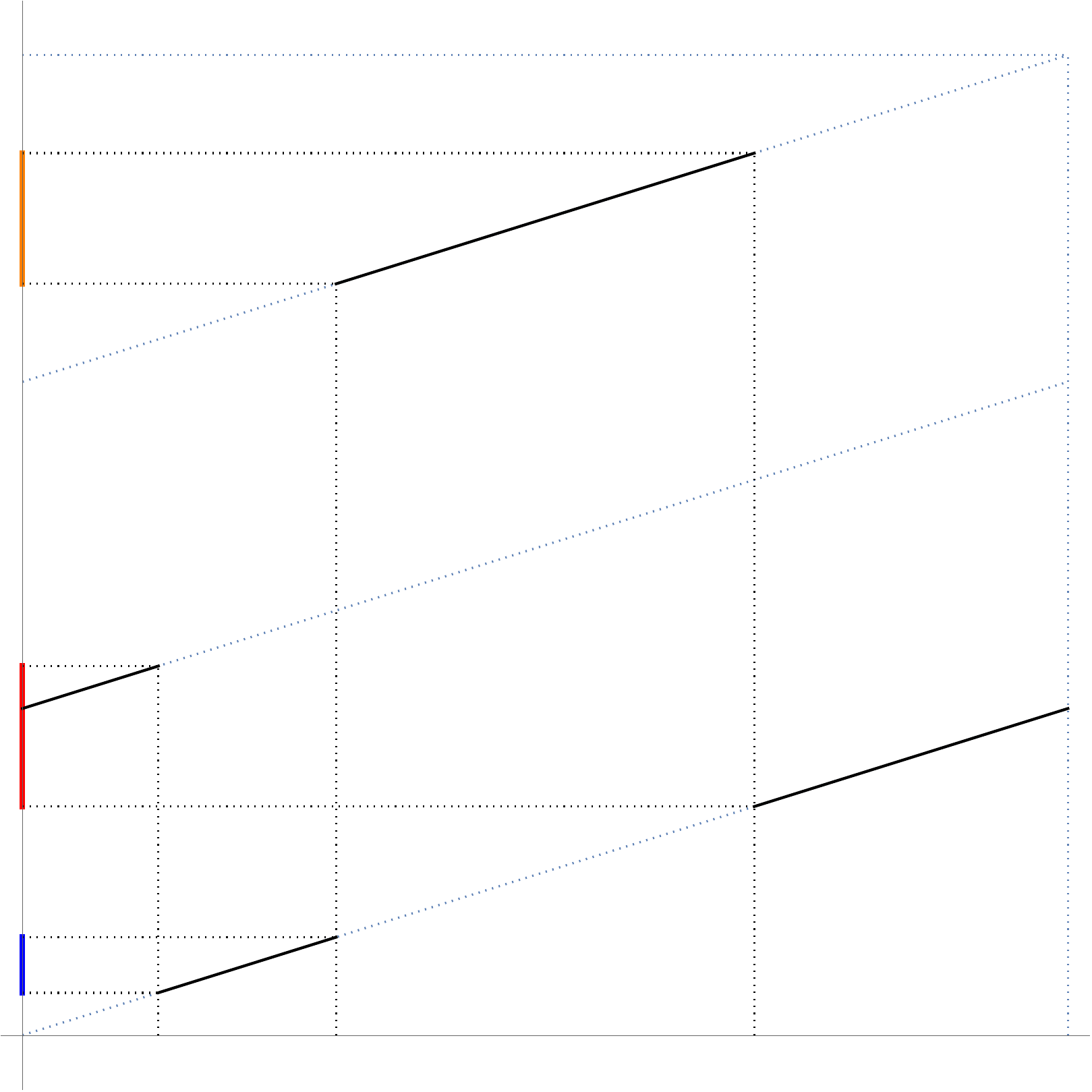}
  \caption{a 3-flower for $E_3$}
  \label{fig:sub2}
\end{subfigure}
\caption{Examples of preimage selectors and corresponding flowers}
\label{fig:flowerexs}
\end{figure}

Every flower for $E_d$ carries at least one $E_d$-invariant probability measure. Indeed for a flower $F$, define the maximal invariant set of $F$
\[
M_F=\bigcap_{n=-\infty}^\infty E_d^{-n}(F).
\]
Note that since $E_d(F)=\T$, we trivially have $M_F=\bigcap_{n=0}^\infty E_d^{-n}(F)$. Observe that by definition of a flower, for every $x\in\T$ and every $n\in \N$, there is always at least one point in $E_d^{-n}(x)$ that falls inside of $F$. 
In addition to this observation, there is at least one point $x \in F$ that has a forward orbit that always remains in $F$, as shown in the following lemma.

\begin{lemma}\label{lem:compactmaxinv}
    For a flower $F$ for the map $E_d$, the set $M_F$ is non-empty and compact.
\end{lemma}
\begin{proof}
    Since $\T$ is compact and $F$ is closed, $F$ is also compact. Consider \[M_1:=F\cap E_d^{-1}(F).\] Since every point $x\in \T$ has at least one preimage under $E_d$ that falls in $F$, we have that $M_1\neq \emptyset$. Similarly, the preimage $E_d^{-1}(F\cap E_d^{-1}(F))=E_d^{-1}(F)\cap E_d^{-2}(F)$ has non-empty intersection with $F$. Thus we obtain that for any $k$
    \[
    M_k:=F\cap E_d^{-1}(F)\cap...\cap E_d^{-k}(F) \neq \emptyset.
    \] Every $M_k$ is compact and they are nested
    \[
    M_1 \supset M_2 \supset...\supset M_K\supset...,
    \] so the intersection $M_F=\bigcap_{n=0}^\infty E_d^{-n}(F)$ is non-empty.
\end{proof}

Since $M_F$ is $E_d$ invariant, Lemma \ref{lem:compactmaxinv} implies that there is an $E_d$ invariant measure whose support is contained in $M_F \subset F$. Note that $M_F$ may contain proper closed invariant sets. For example, in the 1-flower $[0,\frac{1}{2}]$ for the doubling map $E_2$, $M_F$ is $\{0\} \cup \{ \frac{1}{2^n}\}_{n=1}^\infty$ but the minimal invariant set, which supports an $E_2$ invariant measure, is the fixed point $\{0\}$.

Recall from the introduction that in our paper, we take invariance of the set $Y$ to mean $E_d(Y)=Y$. This is stronger than the definition that is sometimes used of $E_d(Y)\subset Y$. The rationale for our definition is that we want to put a measure on the set. So we want to include the fixed point $\{0\}$ as an invariant set for $E_2$ but exclude $\{0\}\cap\{\frac{1}{2^n} \}_{n=0}^\infty$.

While every flower supports at least one invariant measure, in Section \ref{section:complexity} we will prove that there is an upper bound on how many invariant measures are supported in a flower.

\subsection{Complexity}\label{section:complexity}

Complexity of a dynamical system may be defined in several related ways. First, for symbolic dynamics, we have the notion of word complexity, which is thoroughly described for example in \cite{allou} or in Theorem 1.3.13 and Chapter 2 of \cite{lothaire}. 

Let $X_m^+=\{0,....m-1 \}^{\Z_{\geq 0}}$ be the one-sided shift space on $m$ symbols and $\sigma:X_m^+ \to X_m^+$ be the shift map $\sigma(x_0x_1x_2...)=(x_1x_2x_3...)$. A subset $X\subset X^+_m$ that is closed and $\sigma$ invariant is called a subshift. Its language $\mathcal{L}_X$ is the set of all allowable words in the subshift. The set $\mathcal{L}_X(n)$ is the set of all words of length $n$ in the language of $X$.
\begin{defn}\label{complexdef1}
    For a subshift $X\subset X_m^+$, the \emph{complexity function} is \[C_X(n)=\#\mathcal{L}_X(n).\]
\end{defn}
Using other terminology, the topology of the shift space $X^+_m$ is generated by cylinder sets of length $n$ which are sets that fix the first $n$ symbols:
\[
C_{a_0,...,a_{n-1}}=\{(x_i)_{i=0}^\infty | x_0= a_0,...,x_{n-1}=a_{n-1}\}.
\] The complexity function for a subshift $X$ counts the number of cylinder sets of length $n$ that intersect $X$.

A theorem of Morse and Hedlund (\cite{morsehedlund}) shows that any infinite subshift satisfies $C_X(n)\geq n+1$ for all $n$, equivalently that if $C_X(n)\leq n$ then $X$ is finite. As discussed in the introduction, we define a Sturmian subshift to be a subshift that is recurrent and that satisfies $C_X(n) \leq n+1$ for all $n$.

A subsystem of the system $(\T,E_d)$ is a subset $Y \subset \T$ that is closed and satisfies $E_d(Y)=Y$. The complexity of a subsystem of $E_d$ may be defined by considering the coding on $d$ symbols given by the partition of $d-$adic intervals.

\begin{defn}\label{complexdef2}
    For positive integer $d$, the \emph{$d$-adic intervals of length $d^{-n}$} or the \emph{$n^{th}$ $d$-adic intervals} are
    \[
    \left\{\left[\frac{j}{d^n},\frac{j+1}{d^n}\right) | j\in\{0,1,...,d^n-1\}\right\}
    \]
    For the subsystem $(Y\subset \T, E_d)$, the \emph{complexity function} is
    \begin{align*}
    C_Y(n)=\#\{\text{d-adic intervals of length $d^{-n}$ that intersect $Y$}\}.
\end{align*}
\end{defn}

To explain the relationship between Definitions \ref{complexdef1} and \ref{complexdef2}, consider the coding map $\chi_d:\T \to X_d^+$ given by
\[
\chi_d(x)=(x_i)_{i=0}^\infty \hspace{2mm} \text{ where }\hspace{2mm} E_d^n(x) \in [\frac{x_n}{d},\frac{x_{n+1}}{d})
\]
which makes the following diagram commute:
\begin{center}
    \begin{tikzcd}
    X_d^+ \arrow[rr, "\sigma"] && X_d^+ \\
    \T\arrow[u, "\chi_d"] \arrow[rr, "E_d"] &&\T\arrow[u, "\chi_d"] 
\end{tikzcd}
\end{center}
Since we take half open, half closed d-adic intervals, ours is a well defined coding.

Under $\chi_d$, the $d$-adic intervals are mapped to the cylinder sets, therefore counting the number of $d$-adic intervals of length $\frac{1}{d^n}$ that intersect $Y$ is equivalent to counting the number of $n$ words in the language of the corresponding coding $\chi_d(Y)\subset X^+_d$.

Finally we also define:
\begin{defn}
    For an $E_d$ invariant measure $\mu$, the complexity function of the measure is given by
    \[C_\mu(n)=C_{\supp(\mu)}(n).\]
\end{defn}

As discussed in the introduction, Sturmian subshifts can be embedded as subsystems of $E_2$ that lie fully inside a half circle and Sturmian shifts, or equivalently measures supported on Sturmian subsystems of $E_2$ have complexity $C(n)\leq n+1$. We extend this result by showing that any subsystem of $E_d$ that is contained inside a $p$-flower has at most linear complexity.

\begin{thm}\label{thm:complexitybound}
    If $Y\subset\T$ is an $E_d$ subsystem that is supported in a $p$-flower $F$, then its complexity function satisfies $C_Y(n) \leq pn+k$ for some constant $k$.
\end{thm}
\begin{proof}
    Fix $n\in \N$ and suppose we know $C_Y(n)$. To calculate $C_Y(n+1)$, we have
    \[ C_Y(n+1) = \#\Big\{ i \in \{0,...,d^{n+1}-1 \} \Big| [\frac{i}{d^{n+1}},\frac{i+1}{d^{n+1}}) \cap Y \neq \emptyset \Big\}.\]

    Note that the $n+1$-st $d$-adic intervals are the preimages under $E_d$ of the $n$-th $d$-adic intervals. 

Suppose that $I=[\frac{i}{d^{n+1}},\frac{i+1}{d^{n+1}})$ is such that $I \cap Y \neq \emptyset$, i.e. it is counted in $C_Y(n+1)$. 
If $I \cap Y$ is non-empty then so is $E_d(I \cap Y)$. By basic properties of intersection, $E_d(I \cap Y) \subset E_d(I) \cap E_d(Y)$. 
Since $Y$ is $E_d$-invariant, $E_d(I) \cap E_d(Y) =E_d(I) \cap Y$. Therefore since
$I \cap Y \neq \emptyset$ we can conclude that $E_d(I) \cap Y \neq \emptyset$.
Since the image $E_d(I)$ is an $n$-th $d$-adic interval with non-empty intersection with $Y$, it must have been counted in $C_Y(n)$. Therefore every interval counted in $C_Y(n+1)$ must be contained in a preimage of an interval that was counted in $C_Y(n)$.

Now, suppose that $J=[\frac{j}{d^n},\frac{j+1}{d^n})$ is such that $J \cap Y \neq \emptyset$, i.e. it was counted in $C_Y(n)$. 
There are $d$ possible preimages of $J$ that may be counted in $C_Y(n+1)$, namely
\begin{align}\label{preimages}
        E_d^{-1}&\Big([\frac{j}{d^n},\frac{j+1}{d^n})\Big)=\nonumber\\&\Big\{[\frac{j}{d^{n+1}},\frac{j+1}{d^{n+1}}),[\frac{j}{d^{n+1}}+\frac{1}{d},\frac{j+1}{d^{n+1}}+\frac{1}{d}),...,[\frac{j}{d^{n+1}}+\frac{d-1}{d},\frac{j+1}{d^{n+1}}+\frac{d-1}{d}) \Big\}
    \end{align}

   However, by the definition of a flower with $p$ petals, there are only $p$ pairs of points that are in the flower and that are of the form $x+\frac{k_2}{d}, x+\frac{k_1}{d}$ for $k_1 \neq k_2$ in $\{0,...,d-1 \}$. The $p$ pairs are exactly the end points of the flower petals, which come from
    \[
    \lim_{t \to z_i^+}\eta(t) \text{ and } \lim_{t \to z_i^-}\eta(t)
    \]
    for the $p$ discontinuity points $\{z_1,...,z_p \}$ of the preimage selector $\eta$.

Suppose that $J$ contains exactly $\delta_J$ of the $p$ discontinuity points $\{z_1,...,z_p\}$ of the preimage selector $\eta$ associated to the flower.
    Then, there are at most $1+\delta_J$ intervals in (\ref{preimages}) that intersect $Y$ and are therefore counted in $C_Y(n+1)$. Summing over all intervals counted in $C_Y(n)$ we obtain a total of $C_Y(n)+p$ intervals that are preimage intervals of those counted in $C_Y(n)$ and that intersect $Y$. Since we showed that all intervals counted in $C_Y(n+1)$ must come from preimages of those counted in $C_Y(n)$ this concludes that
    \[
    C_Y(n+1) \leq C_Y(n)+p
    \]
    and thus we obtain linear complexity
    \[
    C_Y(n) \leq p(n-1) + C(1).
    \]
\end{proof}

As a corollary, due to results by Cyr and Kra (\cite{cyrkra}), the number of invariant measures supported in a flower is bounded. Note that Cyr and Kra prove the complexity bound in the two sided shift space, while our setting is one-sided. By passing a subshift $Y \to \tilde{Y} \subset \{0,...,d\}^\Z$ to its natural extension, we obtain the following corollary.

\begin{cor}\label{cor:complexitybound}
    Let $F$ be a $p$-flower for $E_d$. Then there are at most $p-1$ ergodic $E_d$ invariant measures supported inside $F$.
\end{cor}

\subsection{Flower-supported is equivalent to Sturmian-like}\label{sec:sturmianlike}

In a recent preprint \cite{gaoshen2}, Gao, Shen and Zhang define a ``Sturmian-like'' property and show that maximizing measures of typical potentials have this property. Moreover, they use Sturmian-like as an intermediate step for proving typicality of periodically supported maximizing measures. We will show that the Sturmian-like property is equivalent to being contained in a flower. First, we reproduce their definition here.

Let $E:\R/\Z \to \R/\Z$ be an expanding, analytic, orientation preserving map. Let $K \subset \R/\Z$ be a non-empty compact set with $E(K)=K$.

\begin{defn}\emph{(Gao-Shen-Zhang \cite{gaoshen2} page 12).}\label{def:sturmianlike}

    \begin{itemize}
        \item A point $x\in K$ is called a \emph{critical value} if $\#\{E^{-1}(x) \cap K\}> 1$. Let $\text{Crit}(K)$ denote the set of all critical points in $K$.
        \item A critical value $x$ is called \emph{regular} is there is a connected neighborhood $U\subset \R/\Z$ of $x$ such that for each of the two connected components $B$ of $U\setminus\{x\}$, there is a continuous map $\eta_B:B\to \T$ such that $E\circ \eta_B$ is the identity on $B$ and for $z\in B\cap K$, $\eta_B(z)$ is the unique point in $E^{-1}(z)\cap K$. Let $\text{Reg}(K) \subset \text{Crit}(K)$ denote the set of all regular critical points in $K$.
        \item The subsystem $E|_K$, or the set $K$ is called \emph{Sturmian-like} if any critical point in $K$ is regular.
    \end{itemize}
\end{defn}
Gao et al. define Sturmian-like and prove results for an expanding map, however since Sturmian-like is purely a topological property, we can again assume that $E$ is linear, i.e. $E_d=d x \mod 1$. For any $E_d$-invariant compact set $K$, the following simple lemmas result from the definitions of critical and regular.

\begin{lemma}\label{lemma:crit}
    The set $\text{Crit}(K)$ of critical points in $K$ is closed.
\end{lemma}
\begin{proof}
    Let $\{x_i\}_{i=0}^\infty$ be a sequence of  critical points in $K$ that converge to a point $x$. Since $K$ is compact, $x \in K$, but we will further show that $x$ is critical. For each critical point $x_i$ there are at least two distinct preimages, call them $p_i, q_i \in E_d^{-1}(x_i)$ that are inside of $K$. By compactness of $K$, by passing to a subsequence we can assume that $p_i \to p \in K$ and $q_i \to q \in K$. Since for every $i$, the $p_i$ and $q_i$ are distinct preimages under $E_d$ of the point $x_i$, they must be bounded apart by at least $1/d$; therefore $p\neq q$.
    
   Since $E_d$ is continuous as a map on the circle, we have that $E_d(p_i) \to E_d(p)$ and $E_d(q_i)\to E_d(q)$. But $E_d(p_i)=E_d(q_i)=x_i$ so since a sequence can only have one limit we have $x_i \to E_d(p)=E_d(q)=x$. Therefore $p$ and $q$ are two distinct preimages of $x$ that are in $K$, so $x$ is a critical point.
\end{proof}

\begin{lemma}\label{lemma:reg}
    The set $\text{Reg}(K)$ of regular critical points in $K$ is discrete.
\end{lemma}
\begin{proof}
    Discrete means that for any $x \in \text{Reg}(K)$, there is a neighborhood $U$ of $x$ such that there are no regular points in the neighborhood other than $x$, i.e. $U\cap \text{Reg}(K)=\{x\}$. The definition of regular provides a neighborhood $U \ni x$ where there are no critical points in $U$ other than $x$, so there certainly can't be any other regular critical points in $U$.
\end{proof}

It turns out that with a mild additional condition on $K$ of all points being non-wandering, the $\eta|_B$ in this definition is part of a preimage selector $\eta$ as in Definition \ref{def:preimageselector}. Recall that a point $x\in K$ is called \emph{non-wandering} for $E_d$ if for every open neighborhood $U \ni x$, there is an integer $n>0$ such that $E_d^n(U) \cap U \neq \emptyset$.

\begin{defn}
    Let $K \subset \T$ be a closed $E_d$-invariant set. We say that $K$ has the \emph{non-wandering property} if the non-wandering set of $E_d|_K$ is $K$. That is, for all $x \in K$ and all $\varepsilon>0$, there is a point $y\in K$ and a positive integer $n>0$ such that $d(x,y)<\varepsilon$ and $d(x,E_d^n(y))<\varepsilon$.
\end{defn}
  The non-wandering property means that every point in $K$ is non-wandering with respect to $K$. This property is satisfied when $K$ is the support of a measure, hence why we call this a mild condition for our purposes.

\begin{lemma}\label{lemma:nonwandering}
   Suppose that $K\subset \T$ is a closed $E_d$-invariant set that has the non-wandering property. If a point $y \in K$ is isolated in $K$, then the point $y$ is periodic and all points in the orbit of $y$ are isolated in $K$.
\end{lemma}
\begin{proof}
   If $y$ is isolated then $\{y\}$ is itself an open set $U$. Thus by taking $U=\{y\}$ in the definition of non-wandering, it is guaranteed that there is some time $n$ when $y$ returns to itself, hence it is periodic.
   
    If $y$ is isolated then by continuity of $E_d$ on the circle, every point in its orbit is also isolated.
\end{proof}

Note that any point that is pre-periodic but not truly periodic is wandering, so if $K$ has the non-wandering property then every pre-periodic point is truly periodic. The next lemma shows that assuming the non-wandering property, regular points have at most two preimages that are contained in $K$, which will be necessary to prove that any Sturmian-like set is contained in a flower.

\begin{lemma}\label{lemma:regupperbound}
    Suppose that $K\subset \T$ is a closed $E_d$-invariant set that has the non-wandering property. Let $y\in K$ be a regular critical point with continuous functions 
    \[
    \eta_-(x)=\frac{x}{d}+\frac{j_1}{d}, \quad \eta_+(x)=\frac{x}{d}+\frac{j_2}{d}
    \]
    on the left and right connected components $B_-,B_+$ respectively of $U\setminus\{y\}$ as defined in Definition \ref{def:sturmianlike}. Then, $j_1\neq j_2$, any preimages of $y$ that are in $K$ must satisfy $y_1=\eta_-(y)$ and $y_2=\eta_+(y)$, and so in particular, $\#\{E_d^{-1}(y)\cap K\}=2$.
\end{lemma}
\begin{proof}
    By definition of being a critical point, $\#\{E_d^{-1}(y) \cap K\}\geq2$, so we prove an upper bound on the number of preimages that are contained in $K$. Let $y \in K$ be a regular critical point and let $\eta_-, \eta_+$ be as defined in the statement of the lemma. Suppose for the sake of contradiction that $y_3 =\frac{y}{d}+\frac{j_3}{d}$ where $j_3$ is not equal to either $j_1$ or $j_2$. Then by the definition of regularity, $y_3$ must be isolated in $K$. By Lemma \ref{lemma:nonwandering}, $y_3$ is periodic, and so is $y$. However that implies that any other preimage of $y$ is pre-periodic, and therefore cannot be contained in $K$. This is a contradiction with $y$ being a critical point. Thus every preimage of $y$ that is contained in $K$ must not be isolated, and therefore the functions $\eta_-,\eta_+$ are such that $j_1\neq j_2$ and the preimages of $y$ that fall in $K$ correspond to the two branches chosen by $\eta_-$ and $\eta_+$.
\end{proof}

With the above lemmas, we may finally prove the following theorem.

\begin{thm}
    Let $E_d:\R/\Z \to \R/\Z$ be the map $E_d(x)=dx \mod 1$. Let $K \subset \R/\Z$ be a non-empty compact set with $E_d(K)=K$ that has the non-wandering property. Then the following are equivalent:
    \begin{enumerate}
        \item\label{1} $K$ is Sturmian-like.
        \item\label{2} $K$ has only finitely many critical points.
        \item\label{3} $K$ is contained in a flower.
    \end{enumerate}
\end{thm}
\begin{proof}
   (\ref{1} $\Leftrightarrow$ \ref{2}): Suppose that $K$ is Sturmian-like. Then the set of regular points is exactly equal to the set of critical points. By Lemma \ref{lemma:crit}, the set of critical points is closed. Therefore, in the Sturmian-like case the set of regular points is closed. By Lemma \ref{lemma:reg}, the set of regular points is discrete. A closed discrete subset of a compact set is finite. Therefore, the set of regular points us finite.

  For the converse direction, suppose that $K$ has finitely many critical values $x_1,...,x_n$. Then there are pairwise disjoint neighborhoods $U_1,..,U_n$ around each critical point such that inside each neighborhood $U_i$, there are no critical values other than $x_i$. In particular every point in $U_i\setminus \{ x_i\}$ has at most one preimage that falls inside of $K$. For any $z \in U_i \setminus \{x_i\}$, let $\eta(z)$ denote the unique point in its preimage that intersects $K$. We claim that $\eta$ is continuous on each component of $U_i\setminus\{x_i\}$. Indeed let $z$ be in one of the components of $U_i\setminus\{x_i\}$ and let $z_n$ be a sequence such that $z_n \to z$. Consider the sequence $\eta(z_n)$. By compactness of $K$, we can pass to a subsequence to get that $\eta(z_{n_k})$ converges to a point $y \in K$. By continuity of $E_d$, we obtain that $E_d(y)=z$. Since $z$ is not critical, we must have $y=\eta(z)$. Therefore $\eta(z_n) \to \eta(z)$ so $\eta$ is continuous.

   ($1 \Leftrightarrow 3$): First let $F$ be a flower and suppose that $K \subset F$. By definition of a flower, $F$ is the image of a preimage selector $\eta$, which has finitely many jump discontinuities $\{x_1,...,x_{p} \}$. As stated in section \ref{subsec:flowers}, the flower $F$ defines a unique preimage selector by taking $\eta$ to be right-continuous. The preimages $E_d^{-1}(x_i)$ of the discontinuity points are the end points of the petals of $F$. By definition of a flower, the $x_i$ are the only points such that $\#(E_d^{-1}(x_i)\cap F)>1$. Thus the $\{x_i\}$ are the only possible critical values. 

    For each $x_i\in K$, the preimage selector $\eta$ defines $\eta_B$ that satisfies the definition of regularity. Indeed let $x_i \in K$ be a critical point. Then for $\varepsilon$ small enough, for each one-sided neighborhood $B=(x-\varepsilon,x)$ or $(x,x+\varepsilon)$, the image of $\eta|_{B}$ is entirely contained in one petal of $F$ and is therefore continuous and satisfies the property that for $y \in B, \eta(y)$ is the unique point in $E_d^{-1}(y)$ that falls inside $F$, and therefore possibly $K$.

    For the converse direction, suppose that $K$ is a set such that $E_d|_K$ is Sturmian-like. Then every critical value of $K$ is regular and there are only finitely many critical values (as proved above since $1 \Rightarrow 2$). 
    Label the critical values $\{ x_1<x_2<...<x_p\}$.

    Note that by Lemma \ref{lemma:regupperbound}, for each critical point $x_i$, the maps $\eta_-$ and $\eta_+$ are distinct and there are only two preimages of $x_i$ that fall in $K$. Therefore, in order to construct a pseudo-preimage selector from the $\eta_\pm$'s, we need only check compatibility between adjacent critical values. To be precise in checking compatibility, we will explicitly label the maps $\eta_-, \eta_+$ for a given $x_i$ by $\eta_{(x_i-\varepsilon,x_i)}$ and $\eta_{(x_i,x_i+\varepsilon)}$.

    For two adjacent critical values $x_i<x_{i+1}$, if $\eta_{(x_i,x_i+\varepsilon)}$ and $\eta_{(x_{i+1}-\varepsilon,x_{i+1})}$ are equal as functions, meaning they have chosen the same branch $\frac{y}{d}+\frac{k}{d}$ then define 
    \begin{equation}\label{case1}
    \eta=\eta_{(x_i,x_i+\varepsilon)}=\eta_{(x_{i+1}-\varepsilon,x_{i+1})} \text{ for all } x \in (x_i,x_{i+1}).
    \end{equation} 
    On the other hand, suppose that $\eta_{(x_i,x_i+\varepsilon)}$ and $\eta_{(x_{i+1}-\varepsilon,x_{i+1})}$ are not equal, for example say that $\eta_{(x_i,x_i+\varepsilon)}=\frac{x}{d}+\frac{l_1}{d}$ and $\eta_{(x_{i+1}-\varepsilon,x_{i+1})}=\frac{x}{d}+\frac{l_2}{d}$ for $l_1 \neq l_2$ in $\{0,...,d-1\}$. Since $K$ is compact and invariant, it is either all of $\T$ or has empty interior. However clearly $K \subsetneq \T$ since for example $\frac{x}{d}+\frac{l_2}{d} \notin K$ for $x \in (x_i,x_i+\varepsilon)$. Therefore $K$ has empty interior and in particular the entire interval $[x_i,x_{i+1}]$ can not be in $K$. Thus there exists a point $y \in (x_i,x_{i+1})$ such that $y\notin K$. 
     Now define \begin{equation}\label{case2}
         \eta=\begin{cases}
        \eta_{(x_i,x_i+\varepsilon)} \text{ for } x \in (x_i,y) \\\eta_{(x_j-\varepsilon,x_j)} \text{ for } x \in (y,x_j)
        
    \end{cases}
     \end{equation}
    Repeating this procedure for each pair of adjacent critical points $x_i,x_j$ defines a map $\eta$ with at most $2p$ jump discontinuities, where $p$ is the number of critical values of $K$. We have $K \subset \eta([0,1))$ which is known to be a flower. This completes the proof
\end{proof}

\section{IETs}\label{sec:IETs}

\subsection{Basic Definitions}
An interval exchange transformation (IET) is a map $T: [0,1) \to [0,1)$ that is defined by length data together with a permutation.

\begin{defn}\label{def:iet}
    Let $l=(l_1,...,l_m)$ be a vector of $m$ positive numbers with $l_1+...+l_m=1$. Let $\tau:\{1,...,m\} \to \{1,...,m\}$ be a permutation on $m$ symbols. Let $I_1=[0,l_1),I_2=[l_1,l_1+l_2),...I_m=[l_1+...+l_{m-1},1)$ be the intervals of length $l_1,...,l_m$ respectively. Then the IET $T:[0,1)\to[0,1)$ defined by $(l, \tau)$ is 
     \[
T(x)=x+\sum_{\tau^{-1}(i)<\tau^{-1}(m)} l_i -\sum_{j<m} l_j \hspace{1mm} \text{ for }\hspace{1mm} x\in I_m
\]
\end{defn}
The IET is the piecewise linear map that permutes the intervals $I_1,...,I_m$ according to $\tau$. 
An IET of $m$ intervals is called an $m$-IET. See Figure \ref{fig:IETex} for an example of a 3-IET. 

Note that a 2-IET with length data $(l_1,1-l_1)$ is simply a rotation by $1-l_1$. Any 3-IET, is either a rotation or has one interval $I_1, I_2, I_3$ on which the IET is the identity in which case it is the first return map of a rotation. 4-IETs are the first example of ``interesting'' IETs.

\begin{figure}
    \centering
    \includegraphics[width=0.6\linewidth]{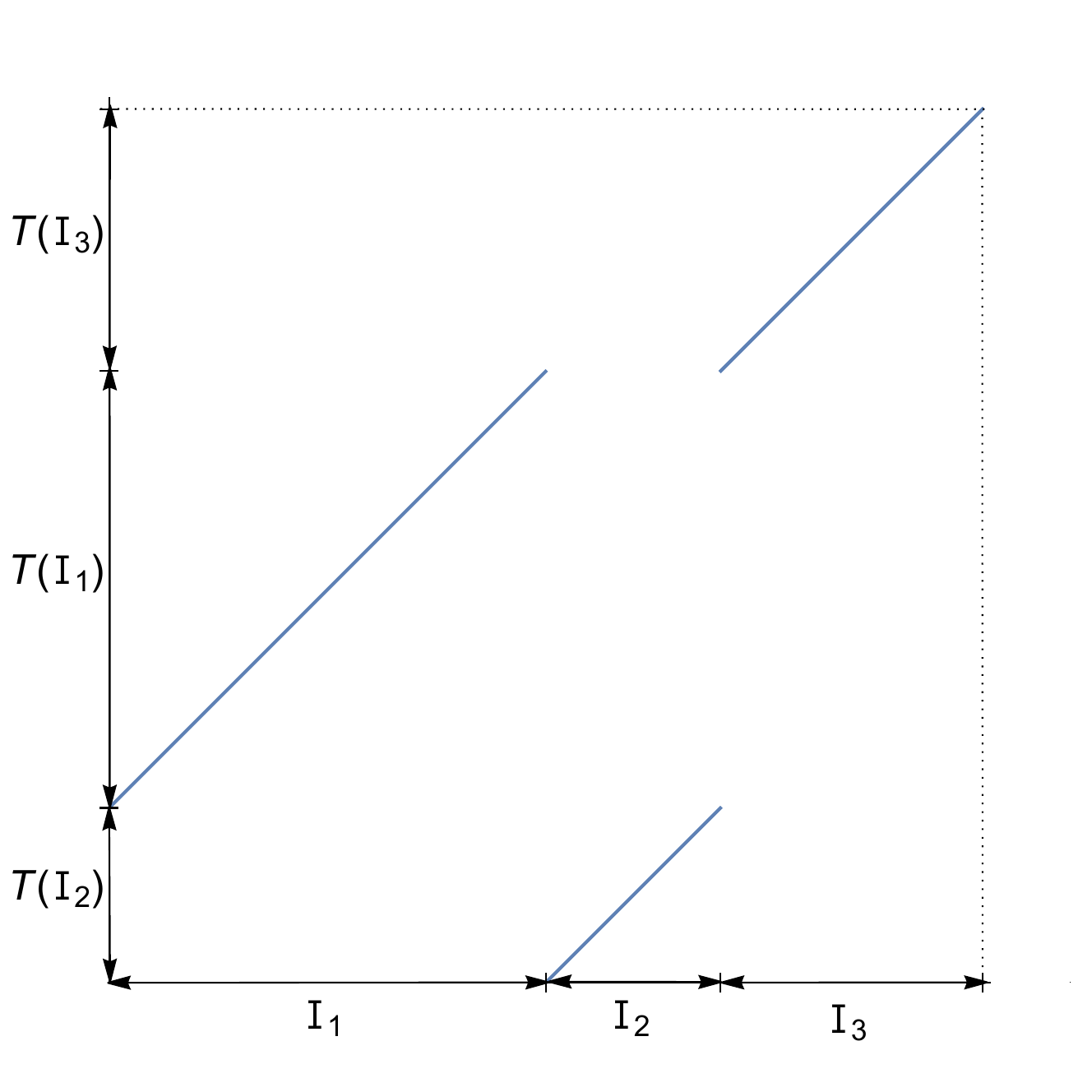}
    \caption{A 3-IET with length data $(0.5, 0.2, 0.3)$ and permutation $1\mapsto2, 2\mapsto1, 3\mapsto3$}
    \label{fig:IETex}
\end{figure}

Any IET has an inverse which is also an IET.

Foundational papers in the study of IETs (\cite{veech1}, \cite{keane}, \cite{yoccoz}), also discuss the following three properties.

\begin{defn}\label{def:irreducible}
    An IET $T$ with length data $(l_1,...,l_m)$ and permutation $\sigma$ is called \emph{irreducible} if there is no $j \leq m$ such that $\sigma(\{1,...,j\})=\{1,...,j\}$. 
\end{defn}
An IET that is not irreducible can be decomposed into IETs on disjoint intervals.

\begin{defn}
    An IET is called \emph{minimal} if every orbit is dense.
\end{defn}
\begin{defn}\label{def:keanes}
    An IET is said to satisfy \emph{Keane's property} if every end point of a starting interval, $0,l_1,l_1+l_2,...,1$ has a distinct infinite orbit.
\end{defn}

It was proved in \cite{keane} that irreducibility plus rational independence of the lengths $l_1,...,l_m$ imply Keane's property and that Keane's property implies minimality. Neither of the converse implications hold, as discussed in section 2 of the paper \cite{ferenczi}.

\subsection{Symbolic Representation}
Any $m$-IET has a natural symbolic representation in $m$ symbols. As in Section \ref{section:complexity}, let $X_m^+:=\{1,...,m\}^{\Z_{\geq0}}$ be the space of one-sided sequences of $m$ symbols and let $\sigma$ be the left shift on $X_m^+$, that is $\sigma(x_0x_1x_2...)=x_1x_2...$. A \emph{subshift} of the one-sided shift space is a closed subset $X \subset X_m^+$ that is invariant under $\sigma$, that is satisfies $\sigma(X)=X$.
\begin{defn}\label{def:naturalcoding}
    For an $m$-IET $T$ with defining intervals $I_1,...,I_m$, the \emph{natural coding} for $T$ is the map $\Phi: [0,1)\to \{0,...,m-1\}^{\Z_{\geq0}}$ that maps $x \mapsto \{x_n\}_{n\geq 0}$ where $T^n(x) \in I_{x_n}$
\end{defn}

The natural coding intertwines the dynamics of $T$ and the shift map on the symbolic space $X_m^+$ as shown in the commuting diagram \ref{diagram:naturalcoding}. Indeed if the itinerary of the orbit of $x\in [0,1)$, with respect to the intervals $I_1,...,I_m$ is $I_{i_1},I_{i_2},I_{i_3},...$ then the itinerary of the orbit of image $T(x)$ is exactly $I_{i_2},I_{i_3},...$, thus $\Phi(T(x))=\sigma(\Phi(x))$.

    \begin{equation}
    \begin{tikzcd}\label{diagram:naturalcoding}
    {X^+_m} \arrow[r, "\sigma"] & {X^+_m}\\
    {[0,1)}\arrow[u, "\Phi"] \arrow[r, "T"] &{[0,1)}\arrow[u, "\Phi"]
\end{tikzcd}
    \end{equation}

For a 2-IET, that is a rotation, the natural coding gives a Sturmian subshift of the shift on 2 symbols.

\subsection{Invariant Measures}
Every IET preserves Lebesgue measure on $[0,1)$. It has been shown that almost every IET is uniquely ergodic (\cite{veech1}), however there are examples of IETs with ergodic measures other than Lebesgue (see Masur \cite{masur} and Veech \cite{veech1}).

\section{Correspondence between IETs and flowers}\label{sec:flowerplusiet}

Fix the dynamics $E_2:\T \to \T$ to be the doubling map $E_2(x)=2x\mod 1$. In this section, we will complete the goal described in Section 1 of developing a correspondence between a special class of IETs and flowers.

\begin{defn}\label{def:deckshuffler}
    A $2m$-IET, $T:[0,1)\to[0,1)$ is called a \emph{deck-shuffler} IET if the $2m$ intervals $A_1<...<A_m<B_1<...<B_m$ with strictly positive lengths are permuted after the first iteration in the order \[T(B_1)<T(A_1)<T(B_2)<T(A_2)<...<T(B_m)<T(A_m)\] where for two intervals $X<Y$ means $x<y$ for all $x\in X$, $y\in Y$. We denote $A:=A_1\cup ...\cup A_m$, $ B:=B_1\cup...\cup B_m$.
\end{defn}

Deck-shuffler IETs are the candidates for IETs that can be realized in the doubling map. 
Note that a deck-shuffler IET always has an even number of intervals and is defined by having a fixed permutation $(1,...,2m)\mapsto (m+1,1,m+2,2,...,2m,m)$. Choosing a deck-shuffler IET is only choosing the length data. Therefore, we will use the following notation:
\begin{notation}
    $T_l$ is the $2m$ deck-shuffler IET with length data $l=(l_1,...,l_{2m})$ where $l$ is chosen from the simplex $\Big\{ l=(l_1,...,l_{2m}) | l_i>0, \sum_{i=1}^{2m}l_i=1\Big\}$. We will use $T_l$ and $T$ interchangeably depending on whether we want to stress the dependence on $l$.
\end{notation}
The fixed permutation ensures that all deck-shuffler IETs are automatically irreducible (Definition \ref{def:irreducible}). It also ensures that the following is true.

\begin{lemma}\label{lem:leftright}
    $T_l|_A$ is strictly increasing and $T_l|_B$ is strictly decreasing. Furthermore for $x\in A$ we have $T_l(x)>x$ and for $x\in B$ we have $T_l(x)<x$. Therefore, $T_l|_A$ and $T_l|_B$ are each order preserving.
\end{lemma}
\begin{proof}
    We may explicitly write:
    \[T_l(x) =\begin{cases}
        x+|B_1|+...+|B_i|  \text{ for } x\in A_i\\
        x-|A_{i}|-...-|A_m|  \text{ for } x\in B_i.
    \end{cases} \]
    Let $x,y \in A$ be such that $x<y$. If $x$ and $y$ are in the same $A_i$ then $T_l$ is clearly an increasing isometry. If $x\in A_i$ and $y\in A_j$ then it must be the case that $i<j$. Therefore 
    \[T_l(y)=y+|B_1|+...+|B_i|+...+|B_j| > y+|B_1|+...+|B_i|>x+|B_1|+...+|B_i|=T_l(x).\] A similar argument shows that $T_l|_B$ is decreasing.
    Furthermore it is clear that when $x \in A$, $T_l(x)$ is $x$ plus a positive number thus $T_l(x) > x$ and when $x\in B$, $T_l(x)$ is $x$ minus a positive number thus $T_l(x)<x$.
\end{proof}

\begin{cor}\label{cor:nofixedpoints}
    A deck-shuffler IET has no fixed points.
\end{cor}

In the next two subsections we will flesh out the correspondance between $2m-1$ flowers and $2m$ deck shuffler IETs. In section \ref{section:flowertoiet}, for a given flower $F$ together with a $E_2$ invariant measure supported in the flower, we define a function $h_\mu$ that is used to construct a deck-shuffler IET. In section \ref{section:iettoflower}, for a given IET $T_l$, we define a function $H_l$ which is used to construct a flower $F_l$ together with a $E_2$ invariant measure supported inside of $F_l$. Explicit examples will be shown in section \ref{sec:examples}.

\subsection{From Flower to IET}\label{section:flowertoiet}
For this section, we introduce the following notation.

\begin{defn}
    Let $F$ be a $2m-1$ flower with petals $P_1,...,P_{2m-1}$ labeled in cyclic order starting from $0$. Let $\mu_i$ be a $E_2$ invariant measure supported inside $F$. Define the function $h_{\mu_i}:[0,1)\to[0,1)$ to be \[
    h_{\mu_i}(x)=\mu_i([0,x])
    \]
\end{defn}

 We may leave out the index $i$ and just write $\mu$ in the future, but recall that by Corollary \ref{cor:complexitybound} there are at most $i=2m-2$ $E_2$ invariant measures supported in $F$. Additionally we introduce the following notation for later: label the endpoints of each petal as
 \[
 P_i=[p_{i,1},p_{i,2}].
 \]
 For each petal $P_i$, denote by $P_i^*$ the half open version \[P_i^*=[p_{i,1},p_{1,2}).\]
 
 The function $h_\mu$ is a cumulative distribution function for probability measure $\mu$ which means the following automatically holds.

 \begin{prop}
     The function $h_\mu$ satisfies the following:
     \begin{itemize}
         \item $h_\mu$ is right-continuous. It is continuous except for at atoms.
         \item $h_\mu$ is onto $[0,1]$
         \item $h_\mu$ is increasing
         \item $h_\mu$ is strictly increasing (and therefore one-to-one) on $\supp(\mu)\subset F$
     \end{itemize}
 \end{prop}

Note additionally the following three results which will be used to prove Theorem \ref{makesiet}. First, Lemma \ref{lem:atomsint} shows that for any flower-supported measure, we can assume that any atoms are in the interior of the flower.

\begin{lemma}\label{lem:atomsint}
   Let $\mu$ be an $E_2$-invariant measure that is supported in some $p$-flower $F$. Then there is a $p$-flower $\tilde{F}$ such that $\supp(\mu) \subset F$ and for any atom $x$ of $\mu$, $x \in int(\tilde{F})$.
\end{lemma}
\begin{proof}
    Let $F$ be a flower containing $\supp(\mu)$. If $\mu$ has no atoms then $\tilde{F}=F$. If any atom $x$ is such that $x \in int(F)$ then $\tilde{F}=F$. Suppose that $\mu(\{x\})>0$ and that the petal $P_i$ contains $x$ with $x \in \partial P_i$.
    
   Since $\mu$ is $E_2$-invariant, $x$ must be a periodic point. Therefore, $x+\frac{1}{2}$ is a pre-periodic point, and cannot be in the support of $\mu$. Since the support of a measure is always closed, there is an $\varepsilon>0$ such that $(x+\frac{1}{2}-\varepsilon,x+\frac{1}{2}+\varepsilon)$ does not intersect $F$. Therefore, if $x=p_{i,1}$, we may define a new flower $\tilde{F}$ by defining the new petal $\tilde{P_i}=P_i\cup [x-\frac{\varepsilon}{2},x]$ and the other petals $\tilde{P_j}=P_j$ for $j\neq i$. Similarly if $x=p_{i,2}$ define $\tilde{P_i}=P_i\cup [x,x+\frac{\varepsilon}{2}]$.
\end{proof}

The next two results show why the assumption $\mu(\{0\})=0$ will be useful.

\begin{lemma}\label{lem:simplify}
    Let $\mu$ be an $E_2$-invariant measure and let $F$ be a $p$-flower such that $\supp(\mu)\subset F$. If $\mu(P_i)=0$ for some petal $P_i$, then $\supp(\mu) \subset \tilde{F}$ where $\tilde{F}$ has $p-2$ petals and $P_i \not\subset \tilde{F}$.
\end{lemma}
\begin{proof}
    We may remove the petal $P_i$ and merge the two petals adjacent to $P_i+\frac{1}{2}$. Indeed let $\tilde{F}=(F\setminus P_i) \cup P_i+\frac{1}{2}$. Note that this $\tilde{F}$ still satisfies the key properties that $\Leb(\tilde{F})=\frac{1}{2}$ and that $int(F)\cap int(F+\frac{1}{2})=\emptyset$ and is therefore still a flower.
\end{proof}

\begin{prop}\label{prop:zero}
    Let $F$ be a flower with $P$ its petal such that $0\in P$. If $\mu(P)>0$ then $\mu(\{0\})>0$.
\end{prop}
\begin{proof}
    Assume for the sake of contradiction that $\mu$ is supported in a flower $F$, that the petal $P$ contains $0$, and that $\mu(P)>0$ but $\mu(\{0\})=0$. By Poincaré Theorem, almost every point in any set of positive measure is recurrent. Therefore, let $x \in P\setminus\{0\}$ be a Poincaré typical point for $\mu$, that is, it is recurrent. Since $x \in \supp(\mu)$ then we must have that its entire forward orbits $\{E_2^n(x) \}_{n\geq0}$ is in $\supp(\mu)$ which is in turn in $F$.

    Since $x\neq 0$ and $0$ is a repelling fixed point, there is some positive integer $k$ such that $E_2^k(x) \notin P$. Let $k$ be the smallest such integer. Since $x$ is recurrent, there is some integer $n>k$ such that $E_2^n(x) \in P$. Let $n$ be the smallest such integer. Then by minimality of $n$ we have that $E_2^{n-1}(x) \notin P$. However note that $E_2^{n-1}(x) \in E_2^{-1}(P)$ since $x\in P$. Therefore, $E_2^{n-1}(x) \in P^c\cap E^{-1}(P)$.

    Since $P$ contains $0$, $E^{-1}(P)$ is exactly two intervals, one of which is strictly contained in $P$ and the other of which is strictly contained in $P+\frac{1}{2}$. Thus if $E_2^{n-1}(x) \in P^c\cap E^{-1}(P)$ we know that $E_2^{n-1}(x) \in P+\frac{1}{2}$. By the definition of a flower $int(P+\frac{1}{2}) \subset F^c$ and so we obtain that $E_2^{n-1}(x) \in F^c$. This is a contradiction with the fact that all points in the forward orbit of $x$ must be contained in $F$. Thus we derive a contradiction and the proposition is proved.
\end{proof}

Assuming an $E_2$-invariant measure does not have an atom at $0$ is a mild assumption since $0$ is a fixed point thus not dynamically interesting. Additionally, Corollary \ref{cor:nofixedpoints} states that there are no fixed points in a deck-shuffler IET, so since the goal is to obtain a deck-shuffler IET from a flower, we must exclude the fixed point of the doubling map.

Additionally, Lemma \ref{lem:simplify} and Proposition \ref{prop:zero} together imply that if $\mu$ is supported in a flower $F$ and $\mu(\{0\})=0$, then $F$ can be simplified by removing the petal containing $0$ without changing the property that $\supp(\mu) \subset F$. Thus we will assume going forward that $0 \notin F$. We now prove our main result of this section.
 
\begin{thm}\label{makesiet}
    Let $\mu$ be a $E_2$ invariant measure supported in a $2m-1$ flower $F$ such that any atoms of $\mu$ fall in the interior of $F$, all petals $P$ of $F$ have positive measure, and $0 \notin F$. Denote $K:=\supp(\mu)\subset F$. Then there is a $2m$ deck-shuffler IET $T_\mu$ that makes the following diagram commute.
    \begin{equation}
    \begin{tikzcd}\label{diagram:2}
    {K} \arrow[r, "E_2"]\arrow[d, "h_\mu"] & {K}\arrow[d, "h_\mu"]\\
    {[0,1)} \arrow[r, "T_\mu"] &{[0,1)}
\end{tikzcd}
    \end{equation}
\end{thm}

\begin{proof}
    Since $0 \notin F$ then $F$ contains $1/2$ and $P_m$ will always be the petal containing $1/2$. Define $T_\mu$ to be the deck shuffler IET with the following intervals:
    \begin{align*}\label{eq:intervals}
        A_1&=h_\mu(P_1^*),...,A_{m-1}=h_\mu(P_{m-1}^*), A_m=h_\mu(P_{m}\cap [0,\frac{1}{2})),\\&B_1=h_\mu(P_{m}\cap [\frac{1}{2},1)),B_2=h_\mu(P_{m+1}^*),...,B_m=h_\mu(P_{2m-1}^*).
    \end{align*}
    Note that because we take $P_i^*$ and because there are no atoms of $\mu$ at any end points $p_{i,j}$, the definitions of the $A_i,B_i$ give left closed/right open intervals as desired in the definition of deck-shuffler IETs.
We will verify that $T_\mu$ defined by these intervals makes diagram \ref{diagram:2} commute.
    
    The petals of the flower satisfy the antipodal condition:
\[ P_{m}^*-\frac{1}{2}<P_1^*<P_{m+1}^*-\frac{1}{2}<...<P_{2m-1}^*-\frac{1}{2}<P_m^*<P_{1}^*+\frac{1}{2}<...<P_{2m-1}^*<P_{m}^*+\frac{1}{2}\]
where the strict inequality is due to choosing the inclusion of endpoints given by the definition of $P_i^*$.
Therefore, applying the doubling map, the petals satisfy
    \begin{align}
        E_2(P_m^*\cap &[\frac{1}{2},1))<E_2(P_1^*)<E_2(P_{m+1}^*)<E_2(P_2^*)<...<E_2(P_m\cap[0,\frac{1}{2})) \label{property:flowerdoubled}
        \\&\text{ and } \hspace{2mm} \bigcup_{i\in\{1,...,2m-1\}} E_2(P_i^*) = [0,1). \label{property:flowerdoubled2}
    \end{align}
    See Figure \ref{fig:flowerdoubled} for an illustration of this property.

    \begin{figure}
        \centering
        \includegraphics[width=0.6\linewidth]{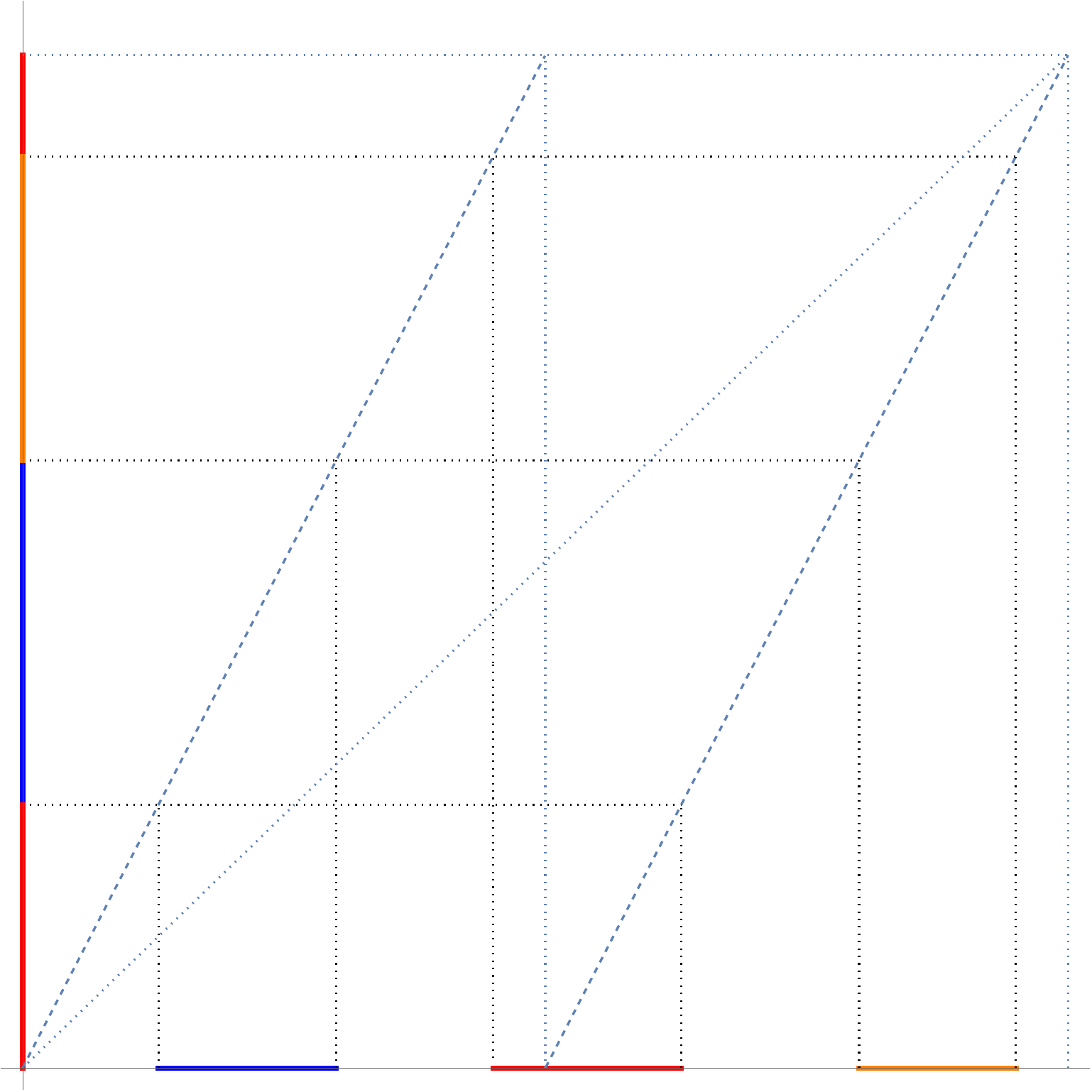}
        \caption{A $3$-flower and its image under $D$}
        \label{fig:flowerdoubled}
    \end{figure}

    For any $x \in K$ we want to verify that $h_\mu(E_2(x))=T_\mu(h_\mu(x))$. First suppose that $x<\frac{1}{2}$ and let $i$ be the index of the petal such that $x \in P_i$. If $x$ is the right endpoint of $P_i$, then it is not an atom and so $h_\mu(x) \in h_\mu(P_i^*)$. Since $x<\frac{1}{2}$ then $h_\mu(P_i^*) = A_i$. By deck shuffling, the explicit definition of $T_\mu$ therefore gives
    \begin{equation}\label{eq:LHS}
        T_\mu(h_\mu(x))=h_\mu(x)+|B_1|+...+|B_i|
    \end{equation}
    where $|B_i|$ is the length of interval $B_i$. By definition of the IET and of $h_\mu$ as a cumulative distribution function, 
    \begin{align*}
        |B_i|=|h_\mu&(P_{m+i-1}^*)|=\mu([0,p_{m+i-1,2}))-\mu([0,p_{m+i-1,1}))\\&=\mu([p_{m+i-1,1},p_{m+i-1,2}))=\mu(P_{m+i-1}^*)=\mu(P_{m+i-1}) \hspace{3mm}\text{ for }\hspace{2mm}i>1
    \end{align*}
    where the last equality is because $p_{m+i-1,2}$ is not an atom. Similarly
    \begin{align*}
        |B_1|=|h_\mu(P_{m}^*\cap [\frac{1}{2},1))|=\mu(P_m\cap[\frac{1}{2},1)).
    \end{align*}

    On the other hand, for $x \in P_i$, property (\ref{property:flowerdoubled}) implies that $E_2(x)>E_2(y)$ for $y \in P_1,...,P_{i-1},P_{m}\cap[\frac{1}{2},1),P_{m+1},...,P_{m+i-2},P_{m+i-1}^*$ and moreover adding property (\ref{property:flowerdoubled2}) to this, 
    \begin{align} \label{eqn:dofx}    
        E_2(x)=|E_2(P_{m}\cap[\frac{1}{2},1))|+|E_2(P_1)|+|E_2(P_{m+1})|+|E_2(P_2)|+...\\+|E_2(P_{m+i-1})|+|E_2([p_{i,1},x])|.\nonumber
    \end{align}
    Here $|\cdot|$ is the usual absolute value. Since $\mu$ is $D$ invariant, applying $h_\mu$ to this equality (\ref{eqn:dofx}), we obtain
    \begin{align*}
       h_\mu&(E_2(x))=
       \\ \mu(&[0,|E_2(P_{m}\cap[\frac{1}{2},1))|+|E_2(P_1)|+|E_2(P_{m+1})|+|E_2(P_2)|+...\\&+|E_2(P_{m+i-1})|+|E_2([p_{i,1},x])|])=
       \\\mu(&[0,|E_2(P_m\cap[\frac{1}{2},1)|])+\mu([|E_2(P_m\cap[\frac{1}{2},1)|,|E_2(P_1)|])+...\\&+\mu([|E_2(P_{m+i-1})|,|E_2([p_{i,1},x])|])=
       \\ \mu(&E_2(P_m\cap[\frac{1}{2},1)))+\mu(E_2(P_1))+\mu(E_2(P_{m+1}))+\mu(E_2(P_2))+...\\&+\mu(E_2(P_{m+i-1}))+\mu(E_2([p_{i,1},x]))=
       \\\mu(&P_m\cap[\frac{1}{2},1))+\mu(P_1)+\mu(P_{m+1})+\mu(P_2)+...+\mu(P_{m+i-1})+\mu([p_{i,1},x])= 
       \\\mu(&[0,x))+\mu(P_{m}\cap[\frac{1}{2},1))+\mu(P_{m+1})+...+\mu(P_{m+i-1}).
    \end{align*}
    
    Since $|B_i|=\mu(P_{m+i-1})$, this is exactly equal to (\ref{eq:LHS}). If $x >1/2$ a similar argument holds this time using that for $h_\mu(x) \in B_i$, the IET definition is $T_\mu(h_\mu(x))=h_\mu(x)-|A_i|-...-|A_{m}|$.
\end{proof}

\subsection{From IET to Flower}\label{section:iettoflower}

\subsubsection{Definitions and Lemmas}

\begin{defn}
    For a deck shuffler $T_l$, the map $\mathcal{H}_l: [0,1) \to \Sigma^+_2 = \{(x_i)_{i=0}^\infty | x_i \in \{0,1\} \}$ is defined by
    \[\mathcal{H}_l(x) = (x_n)_{n \in \Z_{>0}} \text{ where } x_n=\chi_B(T_l^nx).
    \]
\end{defn}

The map $\mathcal{H}_l$ is a coding of the IET in two symbols by the partition $A,B$ instead of the natural coding in $2m$ symbols $\Phi$ as defined in Definition \ref{def:naturalcoding}. Instead, $\mathcal{H}_l$ is a composition of $\Phi$ with the map $\Psi:X_{2m}^+\to X_2^+$ defined by
   \begin{align*}
        \{0,...,m-1\} \mapsto 0 \\
        \{m,...,2m-1\} \mapsto 1.
    \end{align*}
The substitution map $\Psi$ intertwines the dynamics of the two shifts $\sigma_2$ and $\sigma_{2m}$ on $X_2^+$ and $X_{2m}^+$ respectively and sends a cylinder of length $n$ in $X_{2m}^+$ to a cylinder of length $n$ in $X_2^+$. Lemma \ref{lem:codingcommutes} is an obvious consequence of $\mathcal{H}_l$ being defined as a coding map.

\begin{lemma}\label{lem:codingcommutes}
The coding $\mathcal{H}_l$ intertwines the IET $T_l$ and the shift map $\sigma$. That is, the following diagram commtues.
    \begin{equation}
    \begin{tikzcd}\label{diagram:1a}
    {X^+_2} \arrow[r, "\sigma"] & {X^+_2}\\
    {[0,1)}\arrow[u, "\mathcal{H}_l"] \arrow[r, "T_l"] &{[0,1)}\arrow[u, "\mathcal{H}_l"]
\end{tikzcd}
    \end{equation}
\end{lemma}

{The following two lemmas discuss how the dynamics and the coding interact. Lemma \ref{lem:codingperiodic} in particular will be used in the proof of Theorem \ref{thm:H_l}.}

\begin{lemma}\label{lem:noabfixed}
    The image $\mathcal{H}_l$ contains no points that are fixed by the shift $\sigma$. Equivalently, for all $x\in [0,1)$, its orbit under $T_l$ must visit both $A$ and $B$.
\end{lemma}
\begin{proof}
    By Corollary \ref{cor:nofixedpoints}, we know that no points $x\in [0,1)$ are fixed under $T_l$. However, we additionally claim that no points $x\in [0,1)$ can remain only in $A$ or only in $B$ under the dynamics of $T_l$. Indeed, supose that $x\in A$. By Lemma \ref{lem:leftright}, $T_l^{i+1}(x)>T_l^i(x)$ whenever $T^i(x) \in A$. Therefore, since there are only finitely many lengths that can be added by applications of $T_l$, the orbit of $x$ must eventually land in $B$. Similarly, a point $x$ that starts in $B$ must eventually land in $A$. Therefore the sequences $0^\infty$ and $1^\infty$ do not appear in the image of $\mathcal{H}_l$.
\end{proof}

\begin{lemma}\label{lem:codingperiodic}
   A point $x \in [0,1)$ is periodic with respect to $T_l$ with minimal period $n$ if and only if the $A,B$ coding $\mathcal{H}_l(x)$ is a periodic sequence with minimal period $n$.
\end{lemma}
\begin{proof}

     Suppose that $x \in [0,1)$ has periodic $A,B$ coding given by $\mathcal{H}_l(x)=\overline{x_1,...,x_n} \in \{0,1\}^\N$  where $n$ is minimal. Additionally suppose for the sake of contradiction that $T^n(x)\neq x$. By Lemma \ref{lem:noabfixed} we may also suppose without loss of generality that $x\in A$.
    
     Suppose first that $T^n(x)>x$. Consider iterations of $x$ of length of the symbolic cycle period, that is iteration of $n$ steps at a time. Note that in order to preserve the periodic itinerary with respect to $A,B$, we must have that $T^i(x)$ and $T^{n+i}(x)$ are always either both in $A$ or both in $B$. By Lemma \ref{lem:leftright}, $T$ preserves order inside of $A$ and inside of $B$. Therefore since $T^n(x)>x$, after another $n$ iterations, we get that $T^{2n}(x) > T^n(x)$. In general, we obtain a sequence
    \[
    x < T^n(x)<...<T^{n\cdot j}(x)<....
    \]
    Moreover, the distances between each point $T^{n\cdot j}(x)$ in the sequence cannot go to zero since there are only finitely many lengths $l_i$ that $T_l$ may add or subtract at each step. Thus, there is a first time $j$ such that $T^{n\cdot j}(x)$ is in $B$, which contradicts that $T^{n(j-1)}(x)$ and $T^{nj}(x)$ are both in $A$ or both in $B$.

    If we have $T^n(x)<x$ instead, then since there is some $k\in \N$ such that $T^k(x) \in B$, and since $T^n(x), x$ must have the same $A,B$ coding, we have that $T^{n+k}(x)$ and $T^k(x)$ are both in $B$ and the order $T^{n+k}(x) <T^k(x)$ is preserved. Therefore a symmetric argument to before shows that $T^{l(n+k)}(x)$ will eventually land in $A$, and the same contradiction as before arises.
    
    Assuming $T^n(x) \neq x$ derives a contradiction, therefore we must have that $T^n(x)=x$. If there were another integer $m<n$ such that $T^m(x)=x$ then clearly $m$ would also be a period of the symbolic sequence $\mathcal{H}_l$, hence $n$ would not be minimal.

    For the converse direction, suppose that $T^n(x)=x$ and $n$ is the minimal period. Since $x$ is periodic, then it is clearly periodic with respect to the $A,B$ partition. Therefore its image under $\mathcal{H}_l$ is $\overline{x_1...x_m}$ where $m|n$. If $m<n$ then the argument above shows that $x$ must actually be periodic with period $m$, contradicting the minimality of $n$.
        
\end{proof}

Next we define the function $H_l$ which is the coding $\mathcal{H}_l$ composed with the map that takes a sequence in $\{0,1\}^{\Z_{\geq 0}}$ to its binary representation in $[0,1)$.

\begin{defn}\label{def:H_l}
    For a deck-shuffler IET $T_l$, define a map $H_l=H_{T_l}:[0,1)\to[0,1)$ as follows:
\[H_l(x)=\sum_{n=0}^\infty\frac{\chi_B(T_l^nx)}{2^{n+1}}\] where $B=B_1\cup...\cup B_m$.
\end{defn}
The sum is convergent, and for a given $x$, each point in its orbit is either in $A=A_1\cup ...\cup A_m$ or in $B$ because of the choice that intervals are closed on the left and open on the right. Therefore $H_l$ is well defined.

Finally, we know that any IET $T$ has at least one invariant measure and possibly more than one. Recall that in particular Lebesgue measure is always an invariant measure of an IET. The following definition introduces notation for the pushforward measure under $H_l$.
\begin{defn}
    For any invariant measure $\nu_i$ of $T_l$, define $\mu_i:=H_{l*}\nu_i$. Denote the pushforward of Lebesgue as $\mu_l:=H_{l*}\Leb$.
\end{defn}

\subsubsection{Theorems}
Equipped with the above lemmas, we may now present and prove our main theorems. First, Theorem \ref{thm:H_l} describes a complete picture of the structure and key properties of $H_l$.

\begin{thm}\label{thm:H_l}

    The map $H_l$ satisfies the following:
    \begin{enumerate}[label=(\alph*)]
        \item \label{property:intertwine} $H_l$ intertwines the dynamics, ie the following diagram commutes
\begin{equation}
    \begin{tikzcd}\label{diagram:1}
    {[0,1)} \arrow[r, "E_2"] & {[0,1)}\\
    {[0,1)}\arrow[u, "H_l"] \arrow[r, "T_l"] &{[0,1)}\arrow[u, "H_l"]
\end{tikzcd}
    \end{equation}
        \item \label{property:invpushforward} $\mu_l=H_{l*}\Leb$ is $E_2$ invariant.
        \item \label{property:semicont} $H_l$ is right-continuous.
        \item \label{property:inc} $H_l$ is an increasing function.
        \item \label{property:plateaus} 
        The interval $[a,b]$, where $a<b$, is a plateau of $H_l$, that is $[a,b]=H^{-1}_l(\{z\})$, if and only if the measure $\mu_l$ has an atom at $z$. Furthermore, if the interval $[a,b]$ is a plateau, then there exists a natural number $n$ such that $T_l^n(x)=x$ for all $x\in [a,b]$. Conversly, if there is an interval $[a,b]$ with $a<b$ such that $T_l^n(x)=x$ for all $x\in [a,b]$ and some positive integer $n$, then $[a,b]$ is finitely many plateaus.
        \item \label{property:minimal} If $T_l$ is minimal then $H_l$ is strictly increasing. In that case $H_l$ is a right continuous embedding of the deck-shuffler IET $T_l$ into the doubling map $E_2$. 
    \end{enumerate}
\end{thm}

\begin{proof} $ $\newline
\noindent\textit{Proof of \ref{property:intertwine}:} 
    We check that $H_l$ intertwines the dynamics. Indeed:
    \begin{align*}
        E_2(H_l(x))&=2\sum_{n=0}^\infty\frac{\chi_B(T^n(x))}{2^{n+1}}=\sum_{n=0}^\infty\frac{\chi_B(T^n(x))}{2^{n}}=\sum_{n=1}^\infty\frac{\chi_B(T^n(x))}{2^{n}}+\frac{\chi_B(x)}{2^0}=\\
        &\sum_{n=1}^\infty\frac{\chi_B(T^n(x))}{2^{n}} + 0\mod1 = H_l(T(x))
    \end{align*}
    so the diagram commutes. Note as an alternate proof that property \ref{property:intertwine} follows from Lemma \ref{lem:codingcommutes}.
    
    \vspace{1.5mm}

    \noindent\textit{Proof of \ref{property:invpushforward}:} This property now follows from proerty \ref{property:intertwine}. Indeed, more generally, for the pushforward measure $\mu_i=H_{l*}\nu_i$ of any $T_l$ invariant measure $\nu_i$, we have
    \begin{align*}
    E_{2*}\mu_i=E_{2*}H_{l*}\nu_i=(E_2\circ H_l)_*\nu_i&=(H_l\circ T_l)_*\nu_i\\
    &=H_{l*}T_{l*}\nu_i=H_{l*}\nu_i=\mu_i.
\end{align*}

    \vspace{1.5mm}

  \noindent\textit{Proof of \ref{property:semicont}:}
First note that by definition of the IET as permuting left closed, right open intervals (see Definition \ref{def:iet}), the indicator function $\chi_B$ is right continuous and $T_l$ is right continuous. Additionally, for a decreasing sequence $x_i$ with $x_i \to x$, $T_l$ is such that for a large enough $N$, $i>N$ implies that $T_l(x_i)$ is decreasing and approaching $T_l(x)$. Similarly for a fixed iterate $n$ of $T_l$, we have that for $i>N$ for some large enough $N$, the sequence $T_l^n(x_i)$ is decreasing and converges to $T_l^n(x)$. Relabel the sequence $\{x_i\}$ so that it's starting index ensures that $T_l^n(x_i)$ is always decreasing. Then by right continuity of $\chi_B$ we have
\[
\lim_{i\to\infty}\chi_B(T_l^n(x_i))=\chi_B(T_l^n(x)).
\]
Now applying a discrete version of Lebesgue Dominated Convergence Theorem since $\frac{\chi_B(T_l^n(x))}{2^n+1} \leq 1$, we obtain \[
\lim_{i\to\infty}H_l(x_i)=H_l(x)
\]
which proves right continuity.
    
    \noindent\textit{Proof of \ref{property:inc}:} We check that $H_l$ is increasing. Let $x<y$ be two points in $[0,1)$. To understand the order relationship between $H_l(x)$ and $H_l(y)$, we only need to know the positions of the $T_l$ iterates of $x,y$ with respect to the $A,B$ partition. We are interested in the first time $N_0$ such that $T^{N_0}(x)$ and $T^{N_0}(y)$ separate between $A,B$. The codings of $x$ and $y$ by the $A,B$ partition will match up to this point in time. However, to know whether $T^{N_0}(x) \in A$ and $T^{N_0}(y) \in B$, which would indicate $H_l(x)\leq H_l(y)$, or vice versa which would indicate $H_l(x)\geq H_l(y)$, we must carefully keep track of the position of the $x$ and $y$ orbits before the time $N_0$.
    
    To that end, let $N_1 \in \N$ be the first time such that $T^n(x)$ and $T^n(y)$ are not in the same interval out of any of the intervals $A_1,...,A_m,B_1,...,B_m$. 
    Up until and including this point in time the order is still preserved, meaning $T^n(x)<T^n(y)$ for all $n\leq N_1$ because $T$ is piecewise linearly increasing on an individual interval $A_i, B_i$. 
    
    There are three cases for where $T^{N_1}(x)$ and $T^{N_1}(y)$ fall: either both are in an $A$ interval, both are in a $B$ interval, or one is in an $A$ interval and the other is in a $B$ interval.
    First, suppose $T^{N_1}(x)\in A_i$ and $T^{N_1}(y)\in A_j$ for $i \neq j$. Since $T^{N_1}(x)<T^{N_1}(y)$, we must have that $i < j$. Similarly if $T^{N_1}(x) \in B_i$ and $T^{N_1}(y) \in B_j$ it must be that $i<j$. By definition of deck shuffling, $T(A_i)<T(A_j)$ and $T(B_i)<T(B_j)$ for $i<j$, therefore $T^{N_1+1}(x)<T^{N_1+1}(y)$ and more generally, the iterates retain the order $T^n(x) < T^n(y)$ for as long as they remain in the same set $A$ or $B$.

    Thus the situation is reduced to the first time $N_0$ that $T^n(x)$ and $T^n(y)$ separate between $A$ and $B$. (Note that $N_0$ may be $\infty$ if the points never separate on $A,B$.) Since $T^n(x) < T^n(y)$ for all $n<N_0$ and by definition of a deck-shuffler maintaining $T(A_i)<T(A_j)$ and $T(B_i)<T(B_j)$ for $i<j$, it must be the case that $T^{N_0}(x)<T^{N_0}(y)$, so the only possibility at time $N_0$ is that $T^{N_0}(x) \in A$ and $T^{N_0}(y) \in B$. Therefore, applying the coding formula $H_l$, we have
    \begin{equation}\label{formula1}
       H_l(x)=\sum_{n=0}^{N_0-1}\frac{\chi_B(T^n(x))}{2^{n+1}}+\frac{0}{2^{N_0+1}}+\sum_{n=N_0+1}^\infty\frac{\chi_B(T^n(x))}{2^{n+1}} 
    \end{equation}
    and
    \begin{equation}\label{formula2}
        H_l(y)=\sum_{n=0}^{N_0-1}\frac{\chi_B(T^n(x))}{2^{n+1}}+\frac{1}{2^{N_0+1}}+\sum_{n=N_0+1}^\infty\frac{\chi_B(T^n(y))}{2^{n+1}}.
    \end{equation}
    
    Note the first summation in each equation is the same because the itineraries with respect to the $A,B$ partition have matched up to time $N_0-1$. Thus $H_l(x) \leq H_l(y)$, proving that $H_l$ is always increasing. 

    \noindent\textit{Proof of \ref{property:plateaus}:} First we show that a plateau in $H_l$ is equivalent to an atom in $\mu_l$. It is clear that if $H_l$ has a plateau, $[a,b]=H_l^{-1}(\{z\})$, then $\mu_l$ has an atom at $z$ by definition of $\mu_l$ as pushforward of Lebesgue. Conversely, suppose that $z$ is an atom for $\mu_l$. Then it must satisfy that $\mu_l(\{z\})>0$, which is $\Leb(H_l^{-1}(z))>0$. Therefore $H_l^{-1}(z)$ contains two points say $a<b$ (as well as neighborhoods around them). By property \ref{property:inc}, we already know that $H_l$ is increasing. Therefore for any $a<y<b$ we have $H_l(a)\leq H_l(y) \leq H_l(b)$ which is $z\leq H_l(y)\leq z$. Therefore, the entire interval $[a,b]$ satisfies $H_l[a,b]=z$ and we may conclude that $H_l^{-1}(z)$ is an interval, thus $H_l$ has a plateau.
    
    Next we show that if $z$ is an atom in $\mu$ then its associated plateau $[a,b]$ is such that every point $x\in [a,b]$ is periodic with a common period. Let $z$ be an atom of $\mu_l$. Since $z$ is an atom of a $E_2$ invariant measure, it must be a periodic point with finite orbit (that is, a truly periodic point, not a pre-periodic point). Therefore $z$ is a rational number and has a periodic binary representation. This implies that for $x\in H_l^{-1}(z)$, the coding $\mathcal{H}_l(x)$ is a periodic sequence. By Lemma \ref{lem:codingperiodic}, $x$ must be periodic with respect to $T_l$ and with the same minimal period as $\mathcal{H}_l(x)$.

    Lastly suppose that $[a,b]$ is a nontrivial interval such that every $x\in[a,b]$ is periodic under $T_l$ with common minimal period $n$. By Lemma \ref{lem:codingperiodic}, each $x \in [a,b]$ has a periodic coding $\mathcal{H}_l(x)=\overline{x_1,...,x_n}$. Thus $H_l(x)$ is a periodic rational number. There are only finitely many periodic rational numbers with period $n$. Thus the image $H_l([a,b])$ must be finitely many plateaus.
    
    \noindent\textit{Proof of \ref{property:minimal}:} 
    In the proof of property \ref{property:inc}, we see by formulas (\ref{formula1}) and (\ref{formula2}) that equality is achieved if $\chi_B(T^n(y))=1$ for all $n>N_0$ and $\chi_B(T^n(y))=0$ for all $n>N_0$, or if $N_0=\infty$. We will explicitly rule out those cases by using minimality. 
    
    If $T$ is minimal, then the orbits of $x$ and $y$ are both dense in $[0,1)$. Let $\varepsilon=|x-y|$. Let $z=|A_1|+...+|A_m|$ be the end point separating the sets $A$ and $B$ and let $N_2$ be the first time that $|z-T^n(x)|<\varepsilon$. Recall, it is possible a priori that $N_0=\infty$, but $N_2$ is guaranteed to exist by minimality. If $N_2<N_1$, that is the orbits of $x$ and $y$ have stayed in the same interval out of $A_1,...,A_m,B_1,...,B_m$ as each other at each step before time $N_2$, then $|T^{N_2}(x)-T^{N_2}(y)|=|x-y|$ because $T$ is an isometry on each of the intervals. If $N_1<N_2<N_0$ then $T^n(x)$ and $T^n(y)$ have stayed in the same interval out of $A, B$ as each other at each step before time $N_2$. By definition of deck-shuffler, the distance can only increase for two points in the same interval $A$ or $B$ thus $|T^{N_2}(x)-T^{N_2}(y)|\geq|x-y|$. Thus either way, we get that $z-T^{N_2}(x)<\varepsilon$ implies that $T^{N_2}(y)-z>0$ and in particular, $T^{N_2}(y)$ is in $B$. Therefore while $N_0$ could a priori be infinity, minimality implies that $N_0=N_2 < \infty$.
    
    Lastly, minimality also implies that the orbits $T^n(y)$ and $T^n(x)$ can never get stuck in only $A$ or only $B$ and so we can't have $\chi_B(T^n(x))=1$ for all $n>N_0$ and $\chi_B(T^n(y))=0$ for all $n>N_0$. Thus $H_l$ is strictly increasing under the additional assumption that $T$ is minimal. If $H_l$ is strictly increasing, $H_l$ must be one-to-one, and therefore is a right continuous embedding of the IET into the doubling map.
\end{proof}

Next we will show that for any $2m$ deck shuffler $T_l$, the image of its corresponding $H_l$ is contained in a $2m-1$ flower. We introduce some additional notation for use in the next few theorems.
For the deck shuffler IET defined by intervals $A_i, B_i$, label the end points of each interval \[A_i=[a_{i,1},a_{i,2}) \text{ and } B_i=[b_{i,1},b_{i,2}).\] Note that points are labeled twice, e.g. $a_{i,2}=a_{i+1,1}$ for $i<m$.

To show that the image of $H_l$ is contained in a $2m-1$ flower, we will show that the petals of the flower are determined by the $2m-1$ sets
\begin{align}\label{petals}
        \overline{\text{conv}}(H_l(A_1)),...,\overline{\text{conv}}(H_l(A_{m-1}))&,\overline{\text{conv}}(H_l(A_m\cup B_1)),\\\notag&\overline{\text{conv}}(H_l(B_2)),...,\overline{\text{conv}}(H_l(B_m))
    \end{align} 
    where $\overline{\text{conv}}(\cdot)$ denotes the closed convex hull. In particular, we must show that the sets in (\ref{petals}) satisfy the mutually non-antipodal structure of a flower for the doubling map. To that end, we will verify that the following inequality holds:

    \begin{align}\label{convexineqs}
        0 \leq \overline{\text{conv}}(H_l(B_1))-\frac{1}{2}\leq \overline{\text{conv}}(H_l(A_1))\leq\overline{\text{conv}}(H_l(B_2))-\frac{1}{2}\leq  \\\notag \overline{\text{conv}}(H_l(A_2))\leq...\leq \overline{\text{conv}}(H_l(A_m\cup B_1)) \leq \\ \notag \overline{\text{conv}}(H_l(A_1))+\frac{1}{2} \leq\overline{\text{conv}}(H_l(B_2))\leq\overline{\text{conv}}(H_l(A_2))+\frac{1}{2}\leq \\ \notag  \overline{\text{conv}}(H_l(B_3))\leq ...\leq \overline{\text{conv}}(H_l(A_m))+\frac{1}{2}\leq 1
    \end{align}
    where the non-strict inequalities in (\ref{convexineqs}) denotes equality is possible only at the endpoints. Overlapping only at the end points satisfies the mutually non-antipodal property of a flower's petals.

    Note that inequality \ref{convexineqs} is implied by the following strict inequality:
    \begin{align}\label{ineqs}
        0 < H_l(B_1)-\frac{1}{2}< H_l(A_1)<H_l(B_2)-\frac{1}{2}< H_l(A_2)< ...< H_l(A_m)< \frac{1}{2}< \\ \notag H_l(B_1)< H_l(A_1)+\frac{1}{2}< H_l(B_2)< ...< H_l(B_m)< H_l(A_m)+\frac{1}{2}< 1
    \end{align}
    If the inequalities are strict for $H_l(A_i), H_l(B_i)$ then taking the closed convex hull can only create equality at end points.

    We will additionally show that the only way that the image of $H_l$ may fail to be contained in a flower is if three things happen simultaneously: one of the sets $A_i$ for $i \in \{1,...,m-1\}$ (or $B_i$ for $i \in \{2,...,m\}$ or $A_m \cup B_1$) maps under $H_l$ to exactly one point, and both of the inequalities in
    \[\overline{\text{conv}}(H_l(B_i))-\frac{1}{2}\leq \overline{\text{conv}}(H_l(A_i))\leq \overline{\text{conv}}(H_l(B_{i+1}))-\frac{1}{2}\]
    (or in $\overline{\text{conv}}(H_l(A_{i-1}))-\frac{1}{2}\leq \overline{\text{conv}}(H_l(B_i))\leq \overline{\text{conv}}(H_l(A_{i}))-\frac{1}{2}$ or \\ $\overline{\text{conv}}(H_l(B_{m-1}))-\frac{1}{2}\leq \overline{\text{conv}}(H_l(A_m\cup B_{1}))\leq \overline{\text{conv}}(H_l(A_1))+\frac{1}{2}$ respectively) 
    fail to be strict inequalities. Indeed, if these three things happen at once then the petal containing $H_l(A_i)$ (or $H_l(B_i)$ or $H_l(A_m\cup B_1)$) may collapse to a point. To avoid this we introduce the following definition/proposition.

    \begin{prop}\label{prop:noncollapsing}
        Every $2m$ deck-shuffler has the following property, which we call \emph{non-collapsing}: whenever a set  $C\in \{A_1,...,A_{m-1}, A_m\cup B_1, B_2,...,B_m\}$ maps to a single point $\{x \}$ under $H_l$, at least one of the inequalities in (\ref{convexineqs}) adjacent to $\overline{\text{\emph{conv}}}(H_l(C))$ is a strict inequality.       
    \end{prop}

    \begin{proof}
        Suppose that $B_i$ has $i \in \{2,...,m \}$ and is such that $H_l(B_i)=x$. We will show that 
        \[\overline{\text{conv}}(H_l(B_i))<\overline{\text{conv}}(H_l(A_i))+\frac{1}{2} \]
        The other cases when $A_m\cup B_1$ or $A_i$ for $i\in \{1,...,m-1\}$ map to a single point $x$ are similar. Suppose for the sake of contradiction that $\overline{\text{conv}}(H_l(B_i))\cap\overline{\text{conv}}(H_l(A_i))+\frac{1}{2}\neq \emptyset$. This means that at the very least, since $H_l$ is right continuous and $A_i=[a_{i,1},a_{i,2})$, the point $H_l(a_{i,1})=x-\frac{1}{2}$.

        Since $H_l(B_i)=x$, by property \ref{property:plateaus} in Theorem \ref{thm:H_l}, $x$ is a rational number and the $A,B$ coding of $B_i$ is a periodic sequence
        \[
        \mathcal{H}_l(B_i)=\overline{1x_2...x_k}
        \]
        and therefore the coding of $a_{i,1}$ is
        \[
        \mathcal{H}_l(a_{i,1})=0\overline{x_2...x_k1}
        \]
        
        By definition of a deck shuffler, $B_i$ and $A_i$ are mapped in succession after the first iterate. That is $T_l(b_{i,2})=T_l(a_{i,1})$ and $T_l(B_i)<T_l(a_{i,1})$. By our assumption, the itineraries with respect to the $A,B$ partition of $B_i$ and of $a_{i,1}$ match for all time starting at the first iterate. So $B_i$ and $a_{i,1}$ are always either both in $A$ or both in $B$. By Lemma \ref{lem:leftright}, $T_l$ preserves order when restricted to $A$ or $B$. Therefore the order of $B_i$ and $a_{i,1}$ will be preserved for all iterates.

        Property \ref{property:plateaus} in Theorem \ref{thm:H_l} also says that $T^n(B_i)=B_i$ for some positive $n$. Therefore, combining this with the order preservation of $T_l(B_i)<T_l(a_{i,1})$, we have after $n$ iterates that
        \[
        B_i=T_l^n(B_i)<T^n_l(a_{i,1}).
        \]
        By definition of a deck-shuffler, one more iterate will preserve the inequality and insert a positive distance between $T_l^{n+1}(B_i)$ and $T_l^{n+1}(a_{i,1})$. By repeating another $n$ iterates, the distance between $T_l^{jn+1}(B_i)$ and $T_l^{jn+1}(a_{i,1})$ can only grow. Thus we have
        \[
        T_l^{jn+1}(B_i)<[b_{i,2},c_j)<T^{nj+1}_l(a_{i,1})
        \]
        where $\{c_j\}$ is an increasing sequence. Now, since there are only finitely many lengths $l_1,...,l_{2m}$ that can be added at each iterate of $T_l$, we get that there is some $j$ such that $T^{jn+1}_l(a_{i,1}) \in B$, while $T^{jn+1}_l(B_i)$ is still in $A$, which contradicts that $T_l(B_i)$ and $T_l(a_{i,1})$ have the same $A,B$ itinerary.
    \end{proof}

 We may now prove the following theorem.

\begin{thm}\label{thm:containedinflower}
    Let $T_l$ be a $2m$ deck-shuffler IET. Then there is a flower $F_l$ with $2m-1$ petals such that $H_l([0,1))\subset F_l$.
\end{thm}
\begin{proof}
    
    Consider first the case when $H_l$ is strictly increasing. In this case, we can verify the inequality \ref{ineqs} which implies that the sets \ref{petals} are contained inside petals of a $2m-1$ flower.
    
    If $H_l$ is strictly increasing, then we clearly have $H_l(A_1)<...<H_l(A_m)<\frac{1}{2}<H_l(B_1)<...<H_l(B_m)$. 
    Note that by definition,
    \[
    H_l(x)=\frac{1}{2}\chi_B(x)+\frac{1}{2}H_l(Tx).
    \]
    
    Take any $x \in A_i$ and $y \in B_i$. Then $H_l(y)=\frac{1}{2}\cdot1+\frac{1}{2}H_l(Ty)$ and so $H_l(y)-\frac{1}{2}=\frac{1}{2}H_l(Ty)$. Similarly, $H_l(x)=\frac{1}{2}\cdot0+\frac{1}{2}H_l(Tx)=\frac{1}{2}H_l(Tx)$. However, since $T_l$ is a deck shuffler, it is guaranteed that $T(x)>T(y)$. Since $H_l$ is strictly increasing, we therefore have $H_l(Ty)<H_l(Tx)$ which implies $H_l(y)-\frac{1}{2}<H_l(x)$. Thus $H_l(B_i)-\frac{1}{2}<H_l(A_i)$ for any $i\in \{1,...,m \}$. 
    
    Similarly, let $x \in A_i$ and $z \in B_{i+1}$. Then $H_l(z)=\frac{1}{2}\cdot1+\frac{1}{2}H_l(Tz)$ and so $H_l(z)-\frac{1}{2}=\frac{1}{2}H_l(Tz)$, while $H_l(x)=\frac{1}{2}\cdot0+\frac{1}{2}H_l(Tx)=\frac{1}{2}H_l(Tx)$. This time, $T$ being a deck shuffler implies that $T(x) < T(z)$, so $H_l(Tx)<H_l(Tz)$ which implies $H_l(x)<H_l(z)-\frac{1}{2}$. Thus $H_l(A_i)<H_l(B_{i+1})-\frac{1}{2}$ for any $i\in \{1,...,m \}$. 

    By adding 1/2 to the inequalities $H_l(B_i)-\frac{1}{2}<H_l(A_i)$, $H_l(A_i)<H_l(B_{i+1})-\frac{1}{2}$, we obtain the remaining (strict) inequalities in \ref{ineqs}.

    Now, consider the case when $H_l$ is not strictly increasing but just increasing. Then one or more of the ineqaulities in \ref{ineqs} may be not strict. Suppose that $H_l(B_i) \leq H_l(A_i)+\frac{1}{2}$ is not strict. The other cases are similar. Since $H_l$ is increasing the only way the inequality is not strict is that there exists $\alpha \in A_i$ (possibly equal to $a_{i,1}$) and $\beta\in B_i$ such that 
    \begin{equation}\label{eqn:notstrict}
        H_l([\beta,b_{i,2}))=\{ x \}=H_l([a_{i,1},\alpha))+\frac{1}{2}.
    \end{equation}
    
    If the non-strict inequality is isolated, then the image $H_l([0,1))$ is still contained in a $2m-1$ flower because the equality (\ref{eqn:notstrict}) simply implies that $H_l(A_i)$ and $H_l(B_i)$ actually include the end point of the petal containing them.

    However, if two consecutive inequalities in property (1) are not strict, then one of the petals may collapse to a point. Indeed suppose in addition that $H_l(A_{i-1})+\frac{1}{2}\leq  H_l(B_{i})$ is not strict. Again, by increasing behavior of $H_l$, this implies that there is $\alpha' \in A_i$ and $\beta'\in B_{i+1}$ such that 
    \[
    H_l([\alpha',a_{i,2}))+\frac{1}{2}=\{x'\}=H_l([b_{i+1,1},\beta')).
    \]

    If $x=x'$ then the petal containing $H_l(B_i)$ collapses to a point and the petals containing $H_l(A_{i-1}), H_l(A_{i})$ merge into one petal with a hole. The case when both $H_l(B_i)=\{x\}$ is a plateau and $H_l(B_i)-\frac{1}{2} \leq H_l(A_i)$ and $H_l(A_i)\leq  H_l(B_{i+1})-\frac{1}{2}$ are not strict is the only case when the flower structure is at risk.

    However, due to Proposition \ref{prop:noncollapsing}, the behavior that $H_l(B_i)=\{x\}$ is a plateau and $H_l(B_i)-\frac{1}{2} \leq H_l(A_i)$ and $H_l(A_i)\leq  H_l(B_{i+1})-\frac{1}{2}$ are both not strict will never happen. Therefore, even if $H_l$ is not strictly increasing, its image will always be contained in a flower.
\end{proof}

With property \ref{property:plateaus} of $H_l$, bounding the number of plateaus in $H_l$ is equivalent to bounding the number of periodic orbits in the support of $\mu_l$. By Theorem \ref{thm:complexitybound}, a measure supported in a flower has bounded complexity, and so in particular must include only a finite number of periodic orbits in its support. Thus we get the following important result.
\begin{cor}
    Let $T_l$ be a deck-shuffler IET. Then its corresponding $H_l$ is increasing and moreover has finitely many plateaus.
\end{cor}

As is the case with 1-flowers in \cite{bullet}, for a given IET, the image of the associated $H_l$ may be contained in many flowers, particularly if there are finite orbits in the image of $H_l$ because these provide some ``wiggle room'' for the endpoints of the petals. However, with an additional property, the choice of flower becomes unique.

\begin{thm}\label{thm:minimalds}
    Assume that $T_l$ is a $2m$ deck-shuffler IET that has Keane's property. Then there is a unique $2m-1$ flower $F_l$ such that $H_l([0,1))=F_l$
\end{thm}
\begin{proof}
    Theorem \ref{thm:containedinflower} shows that there is at least one $F_l$ containing the image $H_l([0,1))$. In order to show that it is unique, we verify the following condition.

\begin{center}
    \begin{tabular}{l}
 $H_l(a_{i,1})=\lim_{x\to b_{i,2}^-}H_l(x)-\frac{1}{2}$ \text{ for } $1\leq i\leq m$\\
  $\lim_{x\to a_{i,2}^-}H_l(x)=H_l(b_{i+1,1})+\frac{1}{2}$\text{ for } $1\leq i\leq m-1$
\end{tabular}
\end{center}

Firstly, since $a_{m,2}=b_{1,1}$ is the end point separating the partition by $A, B$ by which $H_l$ is coding, it is true that \[\lim_{x \to a_{m,2}^-}H_l(x)<\frac{1}{2}<H_l(b_{1,1}).\] Since $H_l$ by assumption is strictly increasing, no endpoint maps to $\frac{1}{2}$ under $H_l$.

    In order to show that $H_l(a_{i,1})=\lim_{x\to b_{i,2}^-}H_l(x)+\frac{1}{2}$, ie that $H_l(a_{i,1})$ and $\lim_{x\to b_{i,2}^-}H_l(x)$ are antipodal points, we will show that they are preimages of the same point under the doubling map $E_2$. Since $\lim_{x\to b_{i,2}^-}H_l(x) \neq \frac{1}{2}$, then by left continuity of $E_2$,
    \[
    E_2(\lim_{x\to b_{i,2}^-}H_l(x))=\lim_{x\to b_{i,2}^-}E_2(H_l(x)).
    \]
    Since $H_l$ intertwines the dynamics of $E_2$ and $T$, replace
    \[
    \lim_{x\to b_{i,2}^-}E_2(H_l(x))=\lim_{x\to b_{i,2}^-}H_l(T(x)).
    \]
    To bring the limit back inside of $H_l$, we use the Keane condition. By Keane's property, $T(b_{i,2})$ is not equal to any of the original endpoints, and so $H_l$ is left continuous at the limit. Thus we have 
    \[
    \lim_{x\to b_{i,2}^-}H_l(T(x))=H_l(\lim_{x\to b_{i,2}^-}T(x)).
    \]
    Finally, by definition of deck-shuffling, $T(b_{i,2})=T(a_{i,1})$. Therefore we have 
    \[
    H_l(\lim_{x\to b_{i,2}^-}T(x))=H_l(T(a_{i,1}))=E_2(H_l(a_{i,1})).
    \] This completes the proof that $E_2(H_l(a_{i,1}))=E_2(\lim_{x\to b_{i,2}}H_l(x))$. The other equality $E_2(\lim_{x \to a_{i,2}}H_l(x))=E_2(H_l(b_{i+1,1}))$ follows from a similar argument, using commutativity of the dynamics, Keane's condition, and the permutation defined by deck shuffling.
\end{proof}

\subsubsection{Explicit Examples}\label{sec:examples}
We give three explicit examples of a deck-shuffler 4-IET, specified by its length data, and the corresponding measure $\mu_l=H_{l*}\Leb$.
\vspace{2mm}

\noindent \textbf{Example 1}

Let $T_l$ be the deck-shuffler IET with length data $(\frac{2}{5},\frac{1}{5},\frac{1}{5},\frac{1}{5})$. Then every point $x\in [0,1)$ has period $5$, and we have the codings under the $A,B$ partition as shown in Table \ref{tab:example1}.
\begin{table}[h]
    \centering
    \renewcommand{\arraystretch}{1.5}
    \begin{tabular}{|c|c|}
        \hline
        $x \in A_1\cap[0,\frac{1}{5})$ & $H_l(x)=\overline{00011}=\frac{3}{31}$ \\
        $x \in A_1\cap[\frac{1}{5},\frac{2}{5})$ & $H_l(x)=\overline{00110}=\frac{6}{31}$ \\
        $x \in A_2$ & $H_l(x)=\overline{011000}=\frac{12}{31}$ \\
        $x \in B_1$ & $H_l(x)=\overline{10001}=\frac{17}{31}$ \\
        $x \in B_2$ & $H_l(x)=\overline{11000}=\frac{24}{31}$ \\
        \hline
    \end{tabular}
    \caption{$H_l$ for $l=(\frac{2}{5},\frac{1}{5},\frac{1}{5},\frac{1}{5})$}
    \label{tab:example1}
\end{table}

The image under $H_l$ is exactly the orbit $\{\frac{3}{31},\frac{6}{31},\frac{12}{31},\frac{24}{31},\frac{17}{31}\}$. Figure \ref{fig:example1} shows the graph of $H_l$.
\begin{figure}
    \centering
    \includegraphics[width=0.5\linewidth]{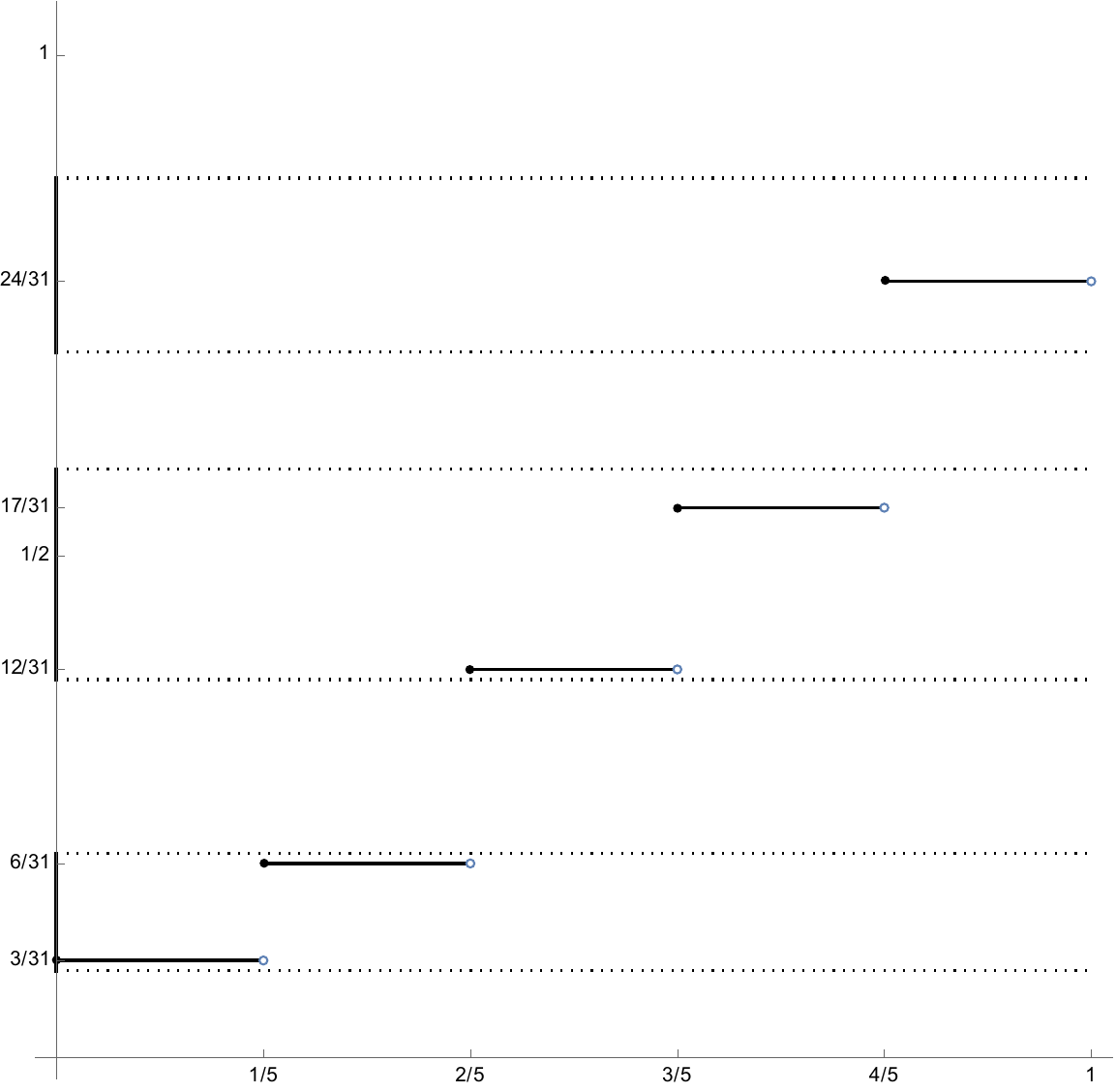}
    \caption{Graph of $H_l$ for $l=(\frac{2}{5},\frac{1}{5},\frac{1}{5},\frac{1}{5})$ and a flower containing the image}
    \label{fig:example1}
\end{figure}

\vspace{2mm}

\noindent \textbf{Example 2}

Let $T_l$ be the deck-shuffler IET with length data $(\frac{3}{10},\frac{2}{10},\frac{2}{10},\frac{3}{10})$. Every point either has period $6$ or has period $4$. Specifically, we have the codings under the $A,B$ partition as shown in Table \ref{tab:example2}. 
\begin{table}[h]
    \centering
    \renewcommand{\arraystretch}{1.5}
    \begin{tabular}{|c|c|}
        \hline
        $x \in A_1\cap[0,\frac{1}{10})$ & $H_l(x)=\overline{000111}=\frac{1}{9}$ \\
        $x \in A_1\cap[\frac{1}{10},\frac{2}{10})$ & $H_l(x)=\overline{0011}=\frac{1}{5}$ \\
        $x \in A_1\cap[\frac{2}{10},\frac{3}{10})$ & $H_l(x)=\overline{001110}=\frac{2}{9}$ \\
        \hline
        $x \in A_2\cap[\frac{3}{10},\frac{4}{10})$ & $H_l(x)=\overline{0110}=\frac{2}{5}$ \\
        $x \in A_2\cap[\frac{4}{10},\frac{5}{10})$ & $H_l(x)=\overline{011100}=\frac{4}{9}$ \\
        \hline
        $x \in B_1\cap[\frac{5}{10},\frac{6}{10})$ & $H_l(x)=\overline{100011}=\frac{5}{9}$ \\
        $x \in B_1\cap[\frac{6}{10},\frac{7}{10})$ & $H_l(x)=\overline{1001}=\frac{3}{5}$ \\
        \hline
        $x \in B_2\cap[\frac{7}{10},\frac{8}{10})$ & $H_l(x)=\overline{110001}=\frac{7}{9}$ \\
        $x \in B_2\cap[\frac{8}{10},\frac{9}{10})$ & $H_l(x)=\overline{1100}=\frac{4}{5}$ \\
        $x \in B_2\cap[\frac{9}{10},\frac{10}{10})$ & $H_l(x)=\overline{111000}=\frac{8}{9}$ \\
        \hline
    \end{tabular}
    \caption{$H_l$ for $l=(\frac{3}{10},\frac{2}{10},\frac{2}{10},\frac{3}{10})$}
    \label{tab:example2}
\end{table}

Thus the image $H_l[0,1)$ splits into the two orbits $\{\frac{1}{5},\frac{2}{5},\frac{4}{5},\frac{3}{5} \}$ and \linebreak $\{\frac{1}{9},\frac{2}{9},\frac{4}{9},\frac{8}{9},\frac{7}{9},\frac{5}{9} \}$. However, both orbits may be contained in one single $3$-flower, as shown in Figure \ref{fig:example2}.
\begin{figure}
    \centering
    \includegraphics[width=0.5\linewidth]{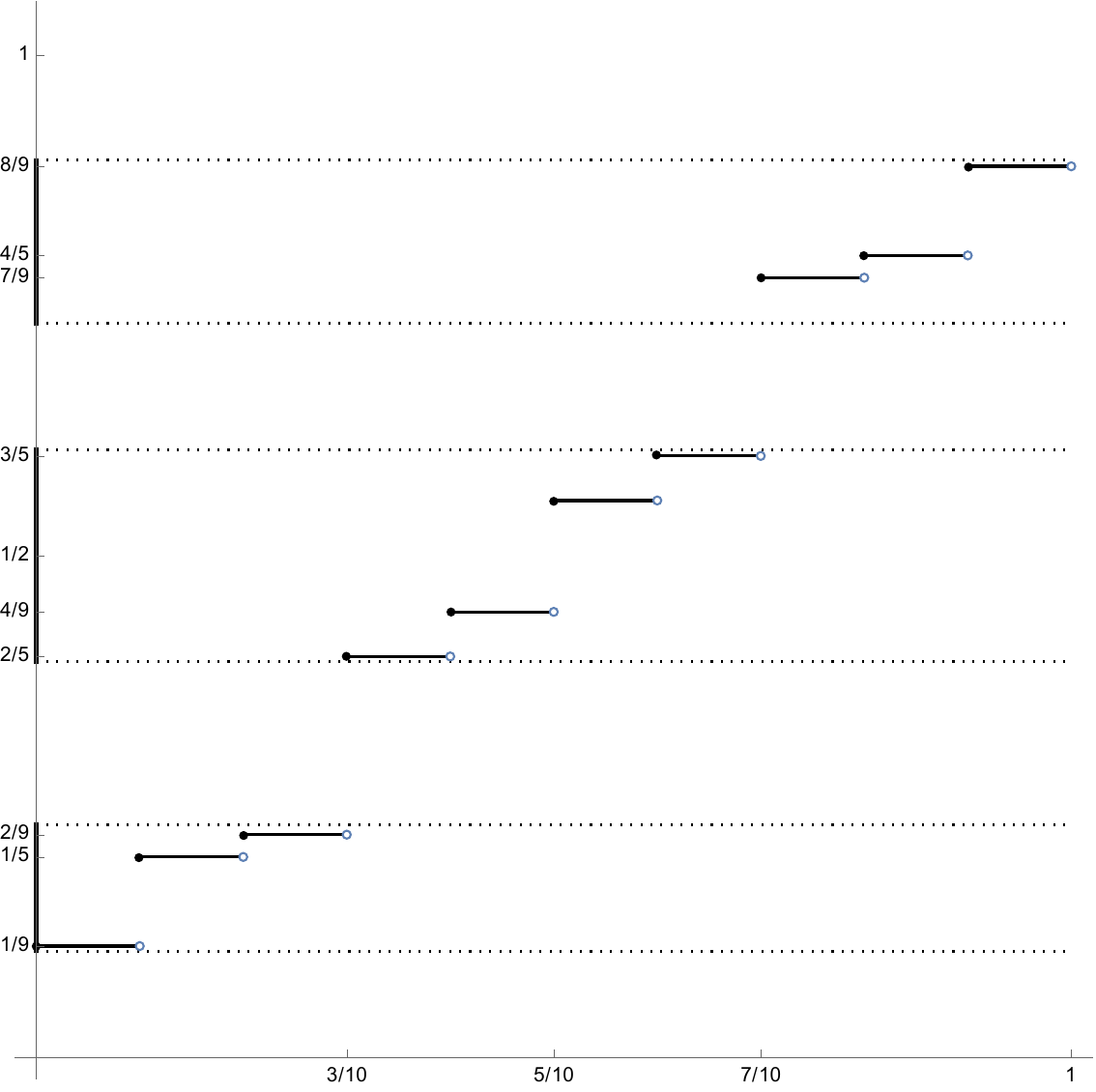}
    \caption{Graph of $H_l$ for $l=(\frac{3}{10},\frac{2}{10},\frac{2}{10},\frac{3}{10})$ and a flower containing the image}
    \label{fig:example2}
\end{figure}

\vspace{2mm}

\noindent \textbf{Example 3}

Let $T_l$ be the deck-shuffler IET with length data $(a,b+\frac{1}{4},b,\frac{1}{4})$ where $a,b \in \R$ are irrational and rationally independent, ie $\frac{a}{b} \notin \Q$. Label $A_{2,1}:= A_2\cap[a,a+b)$ and $A_{2,2}:=A_2\cap [a+b,a+b+\frac{1}{4})$. Observe that, as shown in Figure \ref{fig:example3}, under $T_l$ the intervals have the following behavior:
\begin{align*}
    A_1 &\mapsto [b,a+b) \\
    A_{2,1} &\mapsto B_1\\
    A_{2,2} &\mapsto B_2\\
    B_1 &\mapsto [0,b]\\
    B_2 &\mapsto A_{2,2}\\
\end{align*}

\begin{figure}
    \centering
    \includegraphics[width=0.7\linewidth]{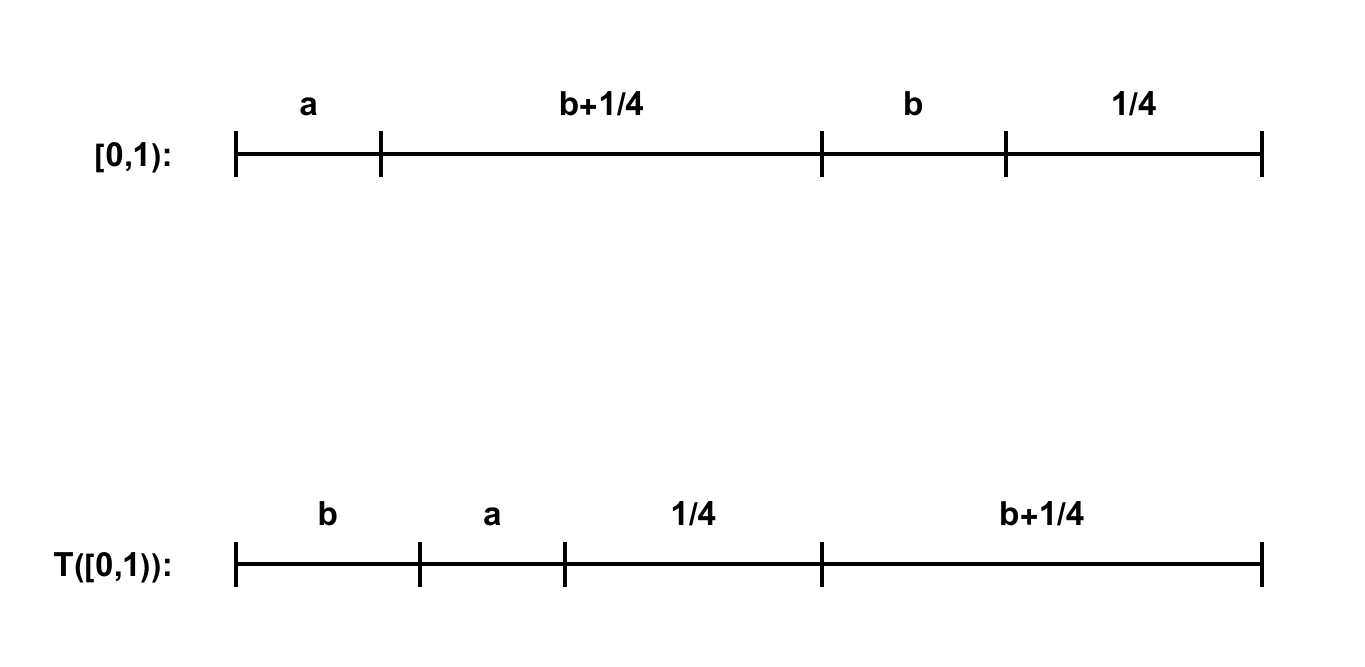}
    \caption{The first iterate of $T_l$ for example 3}
    \label{fig:example3}
\end{figure}

Therefore $A_{2,2}$ and $B_2$ form a periodic component of period 2. Their images under $H_l$ are $\overline{01}=\frac{1}{3}$ and $\overline{10}=\frac{2}{3}$ respectively. Meanwhile, restricted to $A_1\cup A_{2,1}\cup B_1$, the IET $T_l$ acts as an irrational rotation. Indeed, consider the map defined on $A_1, A_{2,1}, B_1$ as
\[
f(x)=\begin{cases}
    2x  \text{ for } x\in A_1\cup A_{2,1}\\
    2x-\frac{1}{4} \text{ for } x \in B_1
\end{cases}
\]

This map $f$ is an invertible map from $A_1 \cup A_{2,1} \cup B_1$ to $[0,1)$. It maps $A_1 \cup A_{2,1}$ to the interval $[0,1-2b)$ and $B_1$ to the interval $[1-2b,1)$. 
The map $f$ intertwines the dynamics of $T|_{A_1 \cup A_{2,1} \cup B_1}$ and the dynamics of rotation by $2b$ on the interval $[0,1)$. Indeed for $x \in A_1$ we have
\[
R_{2b}(f(x))=R_{2b}(2x)=2x+2b \text{ and } f(T_l(x))=f(x+b)=2x+2b
\]
For $x\in A_{2,1}$ we have
\[
R_{2b}(f(x))=R_{2b}(2x)=2x+2b \text{ and } f(T_l(x))=f(x+b+\frac{1}{4})=2x+2b+\frac{1}{2}-\frac{1}{2}=2x+2b
\]
where the second to last equality is because $x+b+\frac{1}{4} \in B_1$. Lastly for $x \in B_1$ we have
\begin{align*}
    &R_{2b}(f(x))=R_{2b}(2x-\frac{1}{2})=2x-\frac{1}{2}+2b-1 \hspace{1mm} \text{ because the rotation is mod 1} \\&\text{and } \hspace{1mm}f(T_l(x))=f(x-a-b-\frac{1}{4})=2x-2a-2b-\frac{1}{2}
\end{align*}
but since $a+2b=\frac{1}{2}$, we have $2x-2a-2b-\frac{1}{2}=2x-a-1=2x-\frac{1}{2}+2b-1$ and so we again get $R_{2b}(f(x))=f(T_l(x))$. This confirms that the following diagram commutes.

\begin{equation}
    \begin{tikzcd}\label{diagramforexample}
    {[0,1)} \arrow[r, "R_{2b}"] & {[0,1)}\\
    {A_1 \cup A_{2,1} \cup B_1}\arrow[u, "f"] \arrow[r, "T_l"] &{A_1 \cup A_{2,1} \cup B_1}\arrow[u, "f"]
\end{tikzcd}
\end{equation}

With this commuting diagram, the $A,B$ coding of a point $x \in A_1 \cup A_{2,1} \cup B_1$ under the action of $T_l$ is exactly the $[0,1-2b),[1-2b,1)$ coding of its image $f(x) \in [0,1)$ under the action of $R_{2b}$. The second coding is known to be the Sturmian coding with rotation number $2b$. As long as $2b$ is irrational, then the coding is an infinite Sturmian.

Therefore, the image $\mathcal{H}_l(A_1 \cup A_{2,1} \cup B_1)$ is an infinite Sturmian subsystem of the shift on two symbols, call it $K$. In other words the image $H_l(A_1 \cup A_{2,1} \cup B_1)$ supports a Sturmian measure with infinite support. By Bullett and Sentenac (\cite{bullet}), the image $H_l(A_1 \cup A_{2,1} \cup B_1)$ is therefore contained in a half circle $[c,c+\frac{1}{2}]$.

Since $H_l(A_{2,2})=\frac{1}{3}$ and $H_l$ is increasing, the endpoint $c$ of the half circle must be less than $\frac{1}{3}$. On the other hand, by Bullett and Sentenac (\cite{bullet}), the half circle can not contain more than one $E_2-$invariant subsystem, so in partiuclar it can not contain both $\frac{1}{3}$ and $\frac{2}{3}$. Thus $c$ must be in $[0,\frac{1}{6})$. 

\begin{figure}
    \centering
    \includegraphics[width=0.5\linewidth]{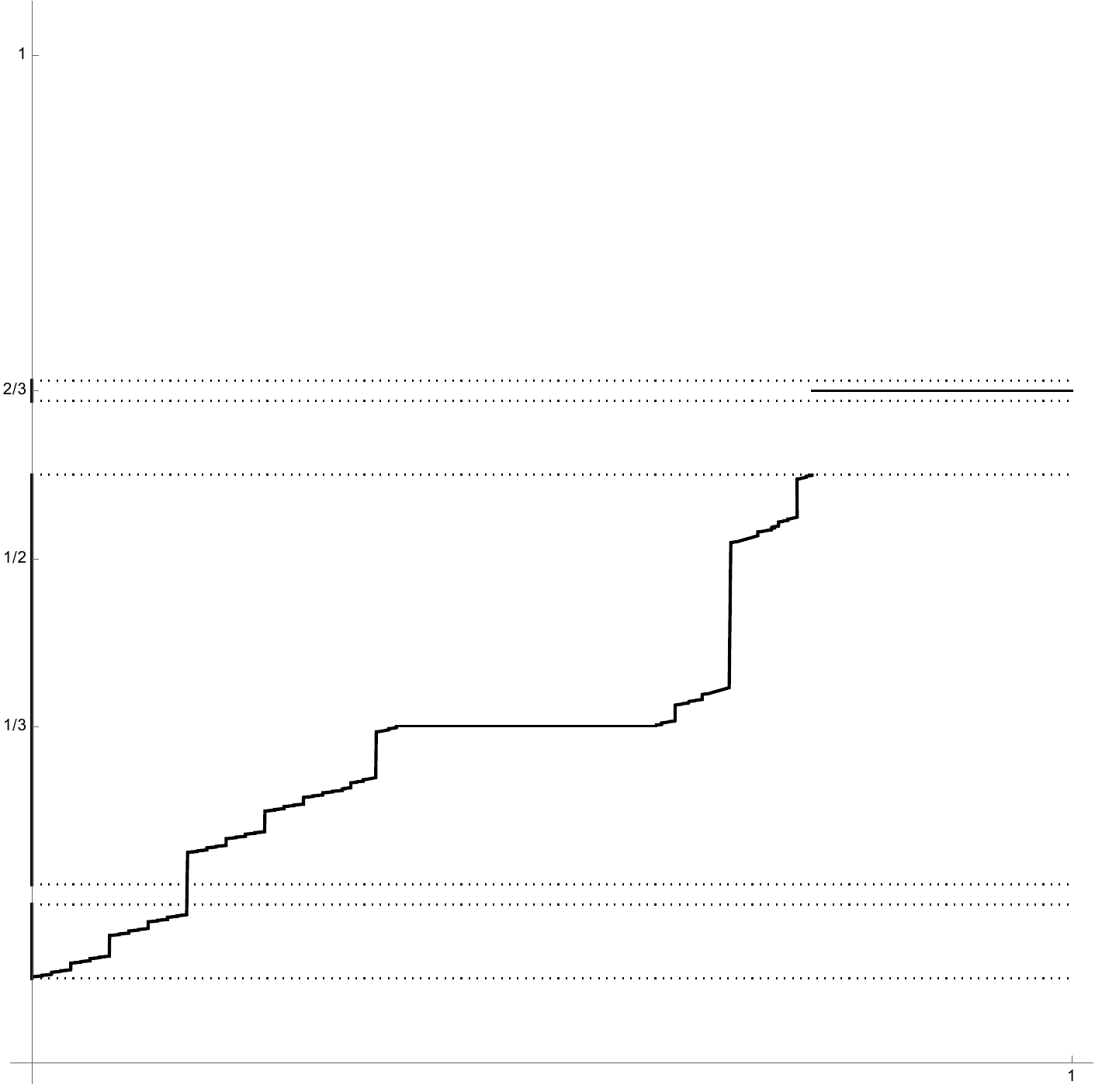}
    \caption{Graph of $H_l$ for $l=(a,b+\frac{1}{4},b,\frac{1}{4})$ where $b \notin \Q$ and a $3-$flower containing the image.}
    \label{fig:example3H}
\end{figure}

Any irrational Sturmian subsystem is closed and does not contain any rational points. Therefore, while $\frac{1}{6}$ is contained in $[c,c+\frac{1}{2}]$, the point $\frac{1}{6}$ is not included in the Sturmian system $K$ itself. Moreover, there exists a small open neighborhood $N=(\frac{1}{6}-\varepsilon,\frac{1}{6}+\varepsilon)$ such that $N\cap K =\emptyset$. Thus a flower $F_l$ containing the image $H_l[0,1)$ is given by $[c,\frac{1}{6}-\varepsilon] \cup [\frac{1}{6}+\varepsilon,c+\frac{1}{2}] \cup [\frac{2}{3}-\varepsilon,\frac{2}{3}+\varepsilon]$. See Figure \ref{fig:example3H} for an example of such an $H_l$ and $F_l$.

\section{Numerical Experiments}\label{sec:ergodicopt}

Recall as introduced in Section \ref{sec:introduction} that ergodic optimization is the study of maximizing measures, that is measures that attain the maximum 
\[
\beta(f)=\sup_{\mu \in \mathcal{M}_S}\int_X f d\mu 
\]
where $S:X \to X$ is a topological dynamical system, $\mathcal{M}_S$ is the set of $S$ invariant measures on $X$, and $f: X \to \R$ is a real valued potential function. A classical example in ergodic optimization was proved by Bousch (\cite{bousch}) for the dynamics of the doubling map on the circle, and the potential functions of degree one trigonometric polynomials. We are interested in extending his results to higher degree trigonometric polynomials.

\begin{defn}
    A degree $n$ trigonometric polynomial is a function of the form $$f(x)=\sum_{k=1}^na _k\cos(k2\pi x)+b_k\sin(k2\pi x).$$
\end{defn}

The $2\pi$ in the argument is to ensure that $f$ is $1$-periodic and can therefore be considered a function on the circle $\T=\R/\Z$. 

We will numerically study maximizing measures of higher degree trigonometric polynomials for the doubling map.
First, note that any trigonometric polynomial is cohomologous with respect to $E_2$ to another trigonometric polynomial containing only trigonometric monomials of odd degree. Thus we will omit all even degree terms in the future. Our conjecture is as follows.

\begin{conjecture}\label{conj}
    Let $n\in \Z_{>0}$ be odd. For the dynamics $E_2:\T\to \T$ and the family of observables $f(x)=a_1\cos(2\pi x)+b_1\sin(2\pi x)+a_3\cos(6\pi x)+b_3\sin(6\pi x)+...+a_n\cos(2n\pi x)+b_n\sin(2n\pi x)$, all maximizing measures are supported in a flower with at most $n$ petals.
\end{conjecture}

If a measure has support equal to a periodic orbit, we may numerically determine whether its support is contained in a flower by computing what we will call its \textit{interlacing number} as follows. First define the \emph{antiorbit} to be the set of antipodal points to the periodic orbit. Then, count how many times there is a switch between a member of the orbit and a member of the antiorbit as you traverse the circle once. Note that the interlacing number of a periodic orbit is always an odd number, and that it is exactly the number of petals of a flower that contains the orbit. 

As an example, the simplest periodic orbit with interlacing number 3 is the period $4$ orbit $\{\frac{1}{5}, \frac{2}{5},\frac{4}{5},\frac{3}{5} \}$, which is represented by the Lyndon word ``$0011$''. By simplest, we mean shortest period and first lexicographically. An image of this orbit together with its antiorbit to indicate its interlacing number is shown in Figure \ref{fig:0011}. Similarly, the simplest periodic orbit with interlacing number 5 is given by Lyndon word ``$0001101$'' (shown in Figure \ref{fig:0001101}). We calculate the interlacing numbers for each of the 2,538 candidate orbits of period up to 14, and the tallies are shown in table \ref{table:allilns}, along with the simplest orbit achieving that interlacing number.

\begin{table}
    \centering
\begin{tabular}{ |c|c|c| }
\hline
Interlacing Number & Tally & Simplest Orbit \\
 \hline
 1 & 65 & 0\\ 
 3 & 470  & 0011 \\ 
 5 & 1,006 &  0001101\\ 
 7 & 742 & 000100111 \\ 
 9 & 227 & 0001011101\\ 
 11 & 28 & 000100111011\\ 
 \hline
\end{tabular}
    \caption{Tally of interlacing numbers for all periodic orbits up to period 14.}
    \label{table:allilns}
\end{table}

\begin{figure}[htbp]
    \centering
    \begin{subfigure}{0.45\textwidth}
        \centering
        \includegraphics[width=\linewidth]{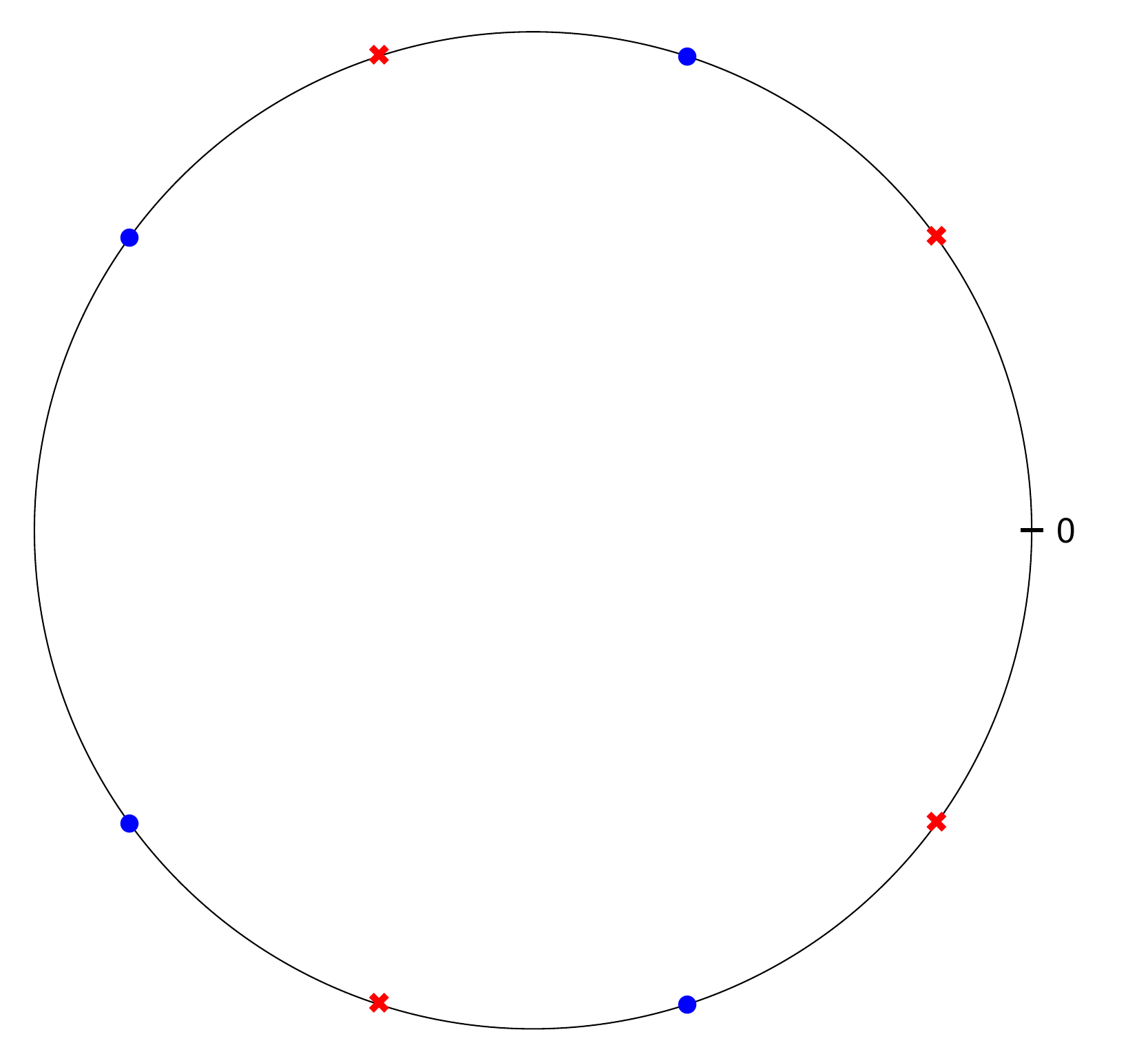}
        \caption{Orbit (blue dots) and antiorbit (red x's) of 0011. Interlacing number is 3.}
        \label{fig:0011}
    \end{subfigure}
    \hfill
    \begin{subfigure}{0.45\textwidth}
        \centering
        \includegraphics[width=\linewidth]{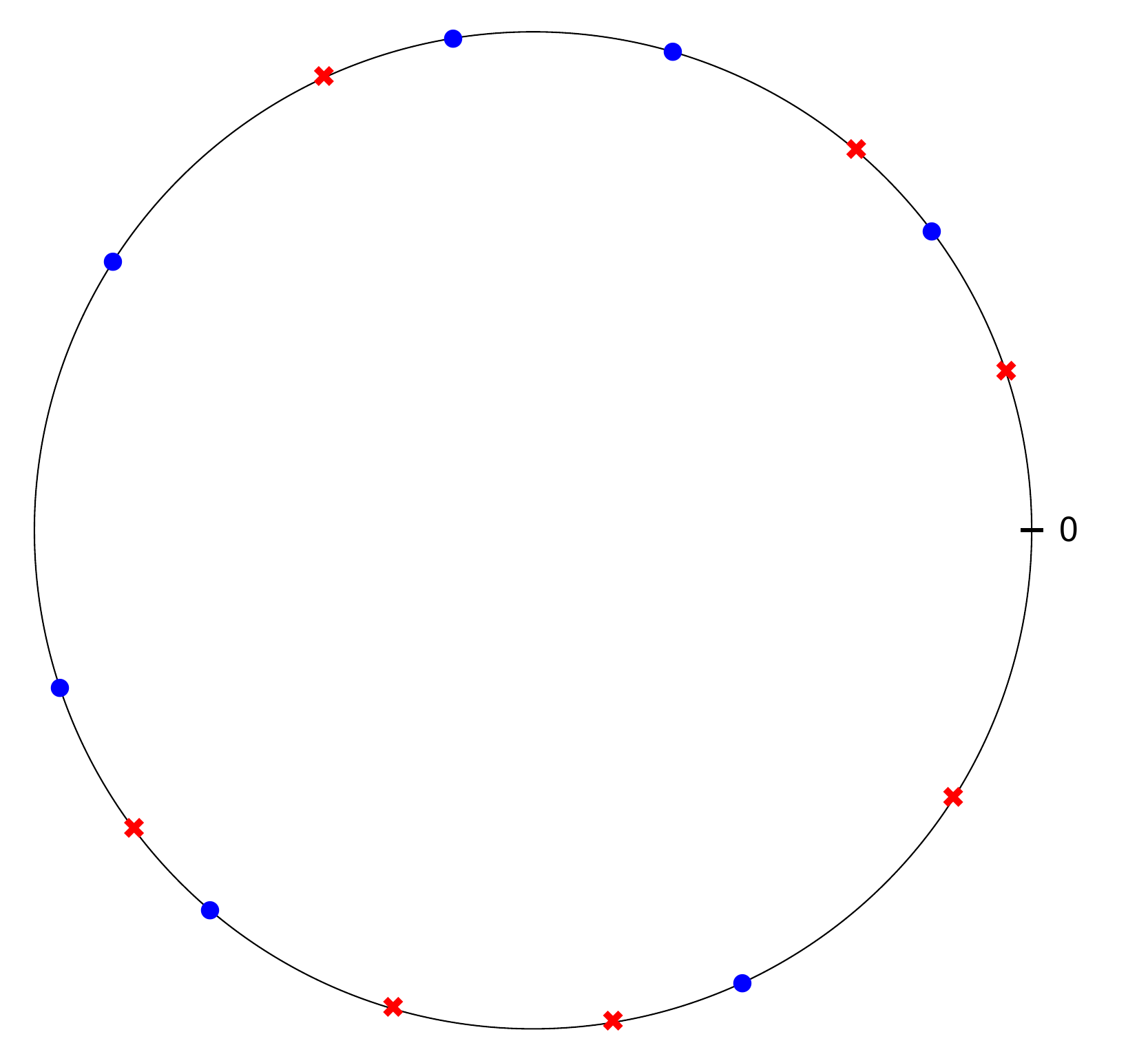}
        \caption{Orbit (blue dots) and antiorbit in (red x's) of 0001101. Interlacing number is 5.}
        \label{fig:0001101}
    \end{subfigure}

    \caption{Visualization of interlacing number}
\end{figure}

To support Conjecture \ref{conj}, we performed a numerical experiment that approximated maximizing measures of degree three trigonometric polynomials and determined whether they were supported in a flower by calculating interlacing numbers. The experiment proceeds as follows. The code produces a randomly generated tuple of coefficients $(a_1,b_1,a_3,b_3)$, selected uniformly from the sphere $S^3$. We then integrate the function \[f(x) = a_1\cos(2\pi x) + b_1\sin(2\pi x)+a_3\cos(6\pi x)+b_3\sin(2\pi x)\]
against every $E_2$-periodic point up to period 14, of which there are 2,538. Out of the resulting list of integrals, the maximum is chosen.

It is generally expected that maximizing measures are supported on periodic orbits and moreover that the orbits have low period. See for example the early experiments of Hunt and Ott in \cite{huntott} showing low period, and the recent work \cite{gaoshen2} showing typical periodicity. It is also generally expected that maximizing measures are unique, see for example \cite{morris}. Therefore, we assume that the maximum obtained from integrating against the list of periodic points up to period 14 is a likely candidate for the maximum, and the corresponding periodic point gets labeled as the support of the \emph{pseudo maximizing measure}.

In our experiment, out of 50,000 randomly generated degree 3 trig polynomials, the pseudo maximizing measures that appeared all had interlacing number 3 or lower. Similarly, out of ten thousand randomly generated degree 5 trig polynomials, the maximizing measures all had interlacing number 5 or lower. See for example the tallies in table \ref{table:iln3s}. 

\begin{table}
    \centering
\begin{tabular}{ |c|c| }
\hline
Interlacing Number & Tally \\
 \hline
 1 &  39,105\\ 
 3 & 10,895   \\ 
 5 & 0 \\ 
 7 & 0 \\ 
 9 & 0 \\ 
 11 & 0 \\ 
 \hline
\end{tabular}
    \caption{Tally of interlacing numbers for orbits that show up as maximizing measures of degree 3 trigonometric polynomials.}
    \label{table:iln3s}
\end{table}

In contrast, we know that the periodic orbits of length up to 14 that were used as candidates have interlacing number up to 11. Since periodic optimization is typical, if Conjecture \ref{conj} were false, it is likely that there would be a periodic counter example of low period. In that case, from a very large sample size of trigonometric polynomials, it is likely that a counterexample would have appeared from our experiment. Since no counterexample appeared, we have strong support for Conjecture \ref{conj}. The code and example results can be found at \url{https://github.com/mbrown60/interlacing_numbers}.
\newpage
\newpage

\printbibliography
\end{document}